\documentclass[12pt]{amsart}
\makeatletter
\let\@@pmod\pmod
\DeclareRobustCommand{\pmod}{\@ifstar\@pmods\@@pmod}
\def\@pmods#1{\mkern4mu({\operator@font mod}\mkern 6mu#1)}
\makeatother
\usepackage{amsmath,amsfonts,amsthm,amssymb,amscd,url,mathrsfs}
\usepackage{color,fullpage,verbatim}

\newcommand{\R}{\mathbb{R}}
\newcommand{\C}{\mathbb{C}}
\newcommand{\Z}{\mathbb{Z}}
\newcommand{\Q}{\mathbb{Q}}
\newcommand{\cE}{\mathcal{E}}
\newcommand{\tE}{\widetilde{\mathcal{E}}}
\newcommand{\cS}{\mathcal{S}}

\DeclareMathOperator{\supp}{supp}
\DeclareFontFamily{OT1}{rsfs}{}
\DeclareFontShape{OT1}{rsfs}{n}{it}{<-> rsfs10}{}
\DeclareMathAlphabet{\mathscr}{OT1}{rsfs}{n}{it}
\DeclareMathOperator{\cond}{cond}

\DeclareMathOperator{\ord}{ord}
\DeclareMathOperator{\lcm}{lcm}

\newtheorem{proposition}{Proposition}[section]
\newtheorem{theorem}{Theorem}[section]

\newtheorem{lemma}{Lemma}[section]
\theoremstyle{remark}
\newtheorem*{remark}{Remark}
\numberwithin{equation}{section}

\title{Subconvexity for modular form $L$-functions \\ in the $t$ aspect}
\author{Andrew R. Booker}
\address{School of Mathematics, University of Bristol, Bristol, BS8 1TW, UK}
\email{andrew.booker@bristol.ac.uk}
\author{Micah B. Milinovich}
\address{Department of Mathematics, University of Mississippi, University, MS
38677 USA}
\email{mbmilino@olemiss.edu}
\author{Nathan Ng}
\address{Department of Mathematics and Computer Science, University of
Lethbridge, Lethbridge, AB Canada T1K 3M4}
\email{nathan.ng@uleth.ca}

\thanks{Research of the first author was supported by EPSRC Grant
\texttt{EP/K034383/1}.  Research of the second author was supported
by the NSA Young Investigator Grants \texttt{H98230-15-1-0231} and
\texttt{H98230-16-1-0311}.  Research of the third author was supported
by an NSERC Discovery Grant. No data were created in the course of this
study.}

\begin{document}
\maketitle

\begin{abstract}
Modifying a method of Jutila, we prove a $t$-aspect subconvexity estimate
for $L$-functions associated to primitive holomorphic cusp forms of
arbitrary level that is of comparable strength to Good's bound for the
full modular group, thus improving on a 36-year-old result.
A key innovation in our proof is a general form of Voronoi
summation that applies to all fractions, even when the level is not
squarefree.
\end{abstract}

\section{Introduction}
Let $f\in S_k(\Gamma_0(N),\xi)$
be a primitive holomorphic cusp form
of weight $k$, level $N$, and nebentypus character $\xi$.
(Here and throughout the paper, ``primitive'' means that $f$ is a
normalized Hecke eigenform in the new space.)
Let 
\[
f(z) = \sum_{n\ge 1} \lambda_f(n) \, n^{(k-1)/2} \, e^{2\pi i n z}
\quad\text{for }\Im(z)>0
\] 
be the normalized Fourier expansion of $f$ at the cusp $\infty$.
The $L$-function associated to $f$ is defined by
\[
L(s,f) = \sum_{n=1}^\infty \frac{\lambda_f(n)}{n^s} = \prod_{p \text{ prime}} \Big( 1- \frac{\lambda_f(p)}{p^s}+\frac{\xi(p)}{p^{2s}} \Big)^{-1}
\]
for $\Re(s)>1$, and by analytic continuation on the rest of $\C$. 

The analogue of the Lindel\"{o}f hypothesis for $L(s,f)$, in the $t$
aspect, is the conjecture that
\[
|L(\tfrac{1}{2}+it,f)| \ll_\varepsilon(1+|t|)^\varepsilon
\quad\text{for any }\varepsilon>0.
\]
A standard application of the Phragm\'en--Lindel\"of principle shows that
\[
|L(\tfrac{1}{2}+it,f)| \ll_\varepsilon(1+|t|)^{\frac12+\varepsilon}.
\]
This is called the convexity estimate for
$L(s,f)$ (in the $t$ aspect), and any improvement on the size of the
exponent on the right-hand side of the inequality is referred to as a
subconvexity estimate.

For $N=1$, Good \cite{Good} showed that
$$
|L(\tfrac{1}{2}+it,f)| \ll|t|^{\frac13} (\log|t|)^{\frac56}
\quad\text{for }|t|\ge2,
$$
using the spectral theory of the Laplacian for the modular group to
estimate so-called shifted convolution sums. Good's approach implicitly
relies on the fact that the Selberg eigenvalue conjecture holds for level
$1$ \cite[Theorem~11.4]{iwaniec}. To generalize it to arbitrary level,
one would have to consider the possibility of exceptional eigenvalues,
which could potentially lead to a weaker estimate. There are situations
where this numerical weakening can be circumvented; for instance,
Lau, Liu, and Ye \cite{LLY} showed, for a related problem, that the
contribution from exceptional eigenvalues can be controlled and causes
no harm to their final result. It is possible that a similar analysis
could be carried out in the present context.

In this paper, we instead consider a subsequent, more elementary,
approach developed by Jutila. Using only Voronoi summation, Farey fractions,
and estimates for exponential integrals, Jutila proved (again for $N=1$)
that
\[
|L(\tfrac{1}{2}+it,f)| \ll_\varepsilon|t|^{\frac13+\varepsilon} 
\quad\text{for }|t|\ge2.
\]
We generalize Huxley's treatment \cite{Huxley} of Jutila's
method and obtain a result for arbitrary level that is essentially as
strong as Good's:
\begin{theorem}\label{mainthm}
Let $f\in S_k(\Gamma_1(N))$ be a primitive cusp form.
Then
\begin{equation*}
|L(\tfrac{1}{2}+it,f)| \ll |t|^{\frac13} \log|t|
\quad\text{for }|t|\ge2,
\end{equation*}
with an implied constant that is polynomial in $k$ and $N$.
\end{theorem}

\begin{remark}
Combining the resolution of the Sato--Tate conjecture \cite{BLGHT}
with general estimates for sums of multiplicative functions due to Shiu
\cite{shiu}, we can marginally improve the inequality in Theorem \ref{mainthm} to
\begin{equation*}
|L(\tfrac{1}{2}+it,f)| \ll |t|^{\frac13} (\log|t|)^{\frac89+\frac8{27\pi}}.
\end{equation*}
However, the implied constant need no
longer be polynomial in $k$ and $N$; see Lemma~\ref{l:L1average}.
One could specify the dependence on $k$ and $N$ more precisely with
additional work, but it seems unlikely to be competitive with recent hybrid
subconvexity bounds for most ranges of the parameters.
\end{remark}

Munshi \cite{Munshi} has recently improved Good's bound for the full modular group, proving that
\[
|L(\tfrac{1}{2}+it,f)| \ll_\varepsilon |t|^{\frac{1}{3}-\frac{1}{1200}+\varepsilon}
\]
for any $\varepsilon>0$ when $N=1$. For $N>1$, prior to this paper, it was known that 
\[
|L(\tfrac{1}{2}+it,f)| \ll_\varepsilon |t|^{\frac{1}{2}-\delta+\varepsilon}
\]
for any $\varepsilon>0$ with $\delta = \frac{1-2\theta}{8}$ for any primitive $f\in S_k(\Gamma_1(N))$ by the work of Wu \cite{wu} and with $\delta=\frac{1-2\theta}{2(3-2\theta)}$ for $k\ge 4$ by the work of Kuan \cite{kuan}. Here $\theta$ is any exponent toward the Ramanujan-Petersson conjecture. Using $\theta=\frac{7}{64}$, we note that $\frac{1-2\theta}{8}=\frac{25}{256}$ and $\frac{1-2\theta}{2(3-2\theta)}=\frac{25}{178}$.

Our main theorem is an instance of what is commonly referred to as a Weyl-type subconvexity estimate
which, in the $t$ aspect, states that $|L(\frac{1}{2}+it)| \ll_\varepsilon
|t|^{\frac{m}{6}+\varepsilon}$ for an $L$-function, $L(s)$, of degree $m$. Classically such estimates are known for the Riemann
zeta-function and Dirichlet $L$-functions.  For degree
2, Good \cite{Good} and Meurman \cite{M} proved
results of this strength for the $L$-functions associated to holomorphic modular
forms and Maass forms on the full modular group. Theorem \ref{mainthm} extends Good's work to arbitrary level (while
the analogous extension for Maass forms remains an open problem). For
primitive $L$-functions of higher degree, obtaining a subconvexity
estimate in the $t$ aspect of Weyl strength remains elusive. Recently Blomer and Mili\'cevi\'c \cite{BM} have
developed a $p$-adic analogue of Jutila's argument to prove a subconvexity
estimate for $L(s,f\times\chi)$ in the character aspect, for a level $1$
form $f$. In an earlier version of this paper, we predicted that our approach could be combined with theirs to prove an analogous  result for general level. Indeed, a result of this type has recently been established by Assing \cite{Assing}.

Our main motivation for establishing Theorem \ref{mainthm} is its use in some applications involving estimates for zeros of $L$-functions. In \cite{BMN}, generalizing a method of Conrey and Ghosh \cite{CG}, we use Theorem \ref{mainthm} to prove quantitative estimates for simple zeros of modular form $L$-functions of arbitrary conductor. Using similar ideas, we can also prove estimates for the number of distinct zeros of $L$-functions. This work is
currently in preparation.

We conclude the introduction with a brief sketch of
the proof. Using an approximate functional equation
for $L(\tfrac12+it,f)$
(Lemma~\ref{Harcos}), we reduce the problem to estimating sums of the
form $\sum_{n=M_1}^{M_2} \lambda_f(n) \, n^{-it}$,
where $M_1\le M_2 \le 2 M_1$. Next, following Jutila, we break the
interval $[M_1,M_2]$ into subintervals on which $n^{-it}$ is well
approximated by additive characters $ce^{2\pi i\alpha n}$,
where $\alpha\in\Q$ has small denominator.
A key novelty in our proof is a generalization of the Voronoi summation
formula (Lemma~\ref{Voronoi}), which applies to all fractions $\alpha$.
Together with a delicate stationary phase analysis
(Proposition~\ref{NathanProp}), we thus transform the additive character
sums into exponential sums that are more complicated but shorter than
those at the start.  Finally, we derive a general large sieve inequality
(Proposition~\ref{Large Sieve}) to convert the problem into a certain
counting problem for Farey fractions (Lemma~\ref{magicmatrices}) that
was solved by Bombieri and Iwaniec.

The outline of the paper is as follows. After some preliminaries
on modular forms in Section~\ref{s:modular}, we prove
Theorem~\ref{mainthm} in broad strokes following the sketch above in
Section~\ref{s:proof}. We defer the most technical parts of the paper, namely the
proofs of Propositions~\ref{NathanProp} and \ref{Large Sieve}, until
Sections~\ref{s:statphase} and \ref{s:largesieve}, respectively.

\section{Modular forms}\label{s:modular}
In this section, we establish some basic properties of modular forms and their $L$-functions
that are
needed in the proof of Theorem~\ref{mainthm}. Throughout this section
we take $f\in S_k(\Gamma_0(N),\xi)$ to be a primitive cusp form with
Fourier coefficients $\lambda_f(n)$,
$\bar{f}\in S_k(\Gamma_0(N),\overline{\xi})$ the dual form
with Fourier coefficients
$\lambda_{\bar{f}}(n)=\overline{\lambda_f(n)}$,
and $\epsilon_f$ the root number of $L(s,f)$, satisfying
$$
\Lambda(s,f)=\epsilon_fN^{\frac12-s}\Lambda(1-s,\bar{f}),
$$
where $\Lambda(s,f)=\Gamma_\C(s+\frac{k-1}2)L(s,f)$ and $\Gamma_{\mathbb{C}}(s) = 2(2\pi)^{-s}\Gamma(s)$.

The first property that we need is a form of `approximate functional
equation' for $L(s,f)$:
\begin{lemma}\label{Harcos}
Let $g:(0,\infty)\to\R$ be a smooth function with functional
equation $g(x)+g(1/x)=1$ and derivatives decaying faster than any negative
power of $x$ as $x\to \infty$. Then
\begin{equation*}
\begin{split}
L(\tfrac{1}{2}+it,f) = \sum_{n=1}^\infty
&\frac{\lambda_f(n)}{n^{\frac12+it}} \, g\!\left(\frac{n}{\sqrt{C}}
\right) + \epsilon_f
\frac{\Gamma_\C(\tfrac{k}{2}\!-\!it)}{\Gamma_\C(\tfrac{k}{2}\!+\!it)}
\sum_{n=1}^\infty \frac{\lambda_{\bar{f}}(n)}{n^{\frac12-it}} \, g\!\left(\frac{ n}{\sqrt{C}} \right) + O_{\varepsilon,g}\!\left( N^{\frac12}C^{-\frac14+\varepsilon} \right)
\end{split}
\end{equation*}
for any $\varepsilon>0$, where $C=C(f,t)$ is the analytic conductor,
defined by
\begin{equation}\label{conductor}
C:=\frac{N}{\pi^2} \Big|\frac{k\!+\!1}{2}+it\Big|
\Big|\frac{k\!+\!3}{2}+it\Big|.
\end{equation}
\end{lemma}
\begin{proof}
This is a special case of a result of Harcos \cite[Theorem 2.5]{Harcos}.
\end{proof}

Next, we need upper estimates for $|\lambda_f(n)|$ on average:
\begin{lemma}\label{l:L1average}
Let $\delta\in\{0,1-\frac{8}{3\pi}\}$, $\alpha\ge0$, $x\ge\frac32$, and
$h\ge1$. Then
\begin{itemize}
\item[(i)]
$\displaystyle \sum_{x<n\le x+h}|\lambda_f(n)|^2\ll_{k,N}
\max(h,x^{\frac35}) \, ;$
\item[(ii)]
$\displaystyle \sum_{x<n\le x+h}|\lambda_f(n)|\ll_{\delta,k,N}
\max(h,x^{\frac35})(\log{x})^{-\delta} \, ;$
\item[(iii)] for $\alpha>1$,
$\displaystyle \sum_{n>x}|\lambda_f(n)| \, n^{-\alpha}\ll_{\delta,\alpha,k,N}
x^{1-\alpha}(\log{x})^{-\delta} \, ;$
\item[(iv)] for $\alpha<1$,
$\displaystyle \sum_{n\le x}|\lambda_f(n)| \, n^{-\alpha}\ll_{\delta,\alpha,k,N}
x^{1-\alpha}(\log{x})^{-\delta}$.
\end{itemize}
Moreover, when $\delta=0$, the implied constants are polynomial
in $k$ and $N$.
\end{lemma}
\begin{proof}
In his work introducing the Rankin--Selberg method,
Rankin \cite{rankin} proved the estimate
$$
\sum_{n\le x}|\lambda_f(n)|^2=A_fx+O_{k,N}(x^{\frac35}),
$$
for a certain explicit $A_f>0$. One can see that both
$A_f$ and the implied constant above grow at most polynomially in $k$
and $N$, and (i) follows.

As for (ii), when $\delta=0$, Cauchy's inequality and (i) imply that
$$
\sum_{x<n\le x+h}|\lambda_f(n)|
\ll_{k,N}\sqrt{h\max(h,x^{\frac35})}
\le\max(h,x^{\frac35}),
$$
again with a polynomial implied constant.
For $\delta=1-\frac{8}{3\pi}$, it follows from the resolution of the
Sato--Tate conjecture \cite{BLGHT} that\footnote{Equality holds in this
estimate when $k>1$ and $f$ does not have CM.
The remaining cases must be handled separately, but are easier to prove
and lead to slightly improved estimates.
Specifically, $8/(3\pi)$ can be replaced by $2/\pi$ for CM forms, at most
$2/3$ for dihedral forms, $5/6$ for tetrahedral forms,
$(5+3\sqrt2)/12$ for octahedral forms, and $(11+6\sqrt5)/30$ for
icosahedral forms.}
$$
\sum_{p\le x}\frac{|\lambda_f(p)|}{p}
\le(1-\delta+o_{k,N}(1))\log\log{x}
\quad\text{as }x\to\infty.
$$
Inserting this into Shiu's estimate \cite[Theorem 1]{shiu},
for any fixed $\varepsilon>0$ we derive that
$$
\sum_{x<n\le x+h}|\lambda_f(n)|\ll_{\varepsilon,k,N} h(\log{x})^{-\delta}
\quad\text{uniformly for }h\ge x^\varepsilon,
$$
which is clearly sufficient for (ii).

Note that (ii) implies
$\sum_{n \le x} |\lambda_f(n)|\ll_{\delta,k,N} x(\log{x})^{-\delta}$.
Using this, a simple exercise in partial summation implies (iii) and (iv). 
\end{proof}

Finally, we require a form of Voronoi/Wilton summation.
As this name is usually understood,
such a formula exists for every fraction $\alpha\in\Q$
only when the level is squarefree.
Since we do not want to impose such a restriction on $f$ in our
hypotheses, we prove a generalization, the basic idea
of which is to replace additive characters by multiplicative characters
at finitely many bad primes.
To this end, for any Dirichlet character $\chi\pmod*{q}$, let
$f^\chi$ denote the unique primitive cusp form
whose Fourier coefficients $\lambda_{f^\chi}(n)$ satisfy
$\lambda_{f^\chi}(n)=\lambda_f(n)\chi(n)$ for all $n$ coprime to $q$;
this is guaranteed to exist by \cite[Theorem~3.2]{AL} and has level
dividing $Nq^2$.

\begin{lemma}
\label{additivetwist}
Let $\alpha=a/q\in\Q$, and define $q^*=\prod_{p\mid q}p^{1+\ord_p{q}}$.
Then
$$
\sum_{n=1}^\infty\frac{\lambda_f(n)e(\alpha n)}{n^s}
=\sum_{\chi\pmod*{q}}
\sum_{m\big|\bigl(\frac{\lcm(Nq,q^2)}{\cond(f^\chi)},q^*\bigr)}
\frac{C(f,\alpha,m,\chi)}{m^s}L(s,f^\chi),
$$
for some numbers $C(f,\alpha,m,\chi)\in\C$ satisfying
$C(f,\alpha,m,\chi)\ll_q 1$.
\end{lemma}
\begin{proof}
Let us first assume that $q=p^e$ is a prime power and $r$ is a
positive integer coprime to $p$. Then the additive twist of
$r^{-s}L(s,f)$ by $\alpha=a/q$ equals
\begin{align*}
&\sum_{n=1}^\infty\frac{\lambda_f(n)}{(rn)^s}e\!\left(\frac{arn}{p^e}\right)
=\sum_{k=0}^{e-1}\frac{\lambda_f(p^k)}{(rp^k)^s}
\sum_{(n,p)=1}\frac{\lambda_f(n)}{n^s}e\!\left(\frac{arn}{p^{e-k}}\right)
+\sum_{k=e}^\infty\frac{\lambda_f(p^k)}{(rp^k)^s}
\sum_{(n,p)=1}\frac{\lambda_f(n)}{n^s}\\
&=\sum_{k=0}^{e-1}\frac{\lambda_f(p^k)}{(rp^k)^s}\sum_{\chi\pmod*{p^{e-k}}}
\frac{\tau(\overline{\chi})\chi(ar)}{\varphi(p^{e-k})}
\sum_{(n,p)=1}\frac{\lambda_f(n)\chi(n)}{n^s}
+\sum_{k=e}^\infty\frac{\lambda_f(p^k)}{(rp^k)^s}
\sum_{(n,p)=1}\frac{\lambda_f(n)}{n^s}\\
&=\sum_{k=0}^{e-1}\frac{\lambda_f(p^k)}{(rp^k)^s}\sum_{\chi\pmod*{p^{e-k}}}
\frac{\tau(\overline{\chi})\chi(ar)}{\varphi(p^{e-k})}
E_{f^\chi,p}(p^{-s})L(s,f^\chi)
\\
&\qquad \qquad+\left(r^{-s}-E_{f,p}(p^{-s})\sum_{k=0}^{e-1}\frac{\lambda_f(p^k)}{(rp^k)^s}\right)
L(s,f),
\end{align*}
where $E_{f,p}$ and $E_{f^\chi,p}$ denote the Euler factor polynomials
of $f$ and $f^\chi$ at $p$, respectively.
Note that this is a linear combination of terms of the form
$(rp^j)^{-s}L(s,f^\chi)$ for characters $\chi\pmod*{p^e}$.

In the general case, by partial fractions, we may express
$\alpha$ as a sum of fractions of the form $a/p^e$, and
applying the prime power case inductively yields a
decomposition of the required type. The estimate
$C(f,\alpha,m,\chi)\ll_q 1$ follows from the
fact that the coefficients
in the above polynomials depend only on local data of $f^\chi$,
together with Deligne's bound.
It remains only to be seen
that the values of $m$ that occur must divide
$\bigl(\frac{\lcm(Nq,q^2)}{\cond(f^\chi)},q^*\bigr)$, which reduces to the
following two assertions in the prime power case:
\begin{equation}
\label{cbound1}
k+\deg{E_{f^\chi,p}}\le e+1
\end{equation}
and
\begin{equation}
\label{cbound2}
k+\deg{E_{f^\chi,p}}+\ord_p\cond(f^\chi)
\le e+\max(e,\ord_p{N}).
\end{equation}

Since $k\le e-1$ and $\deg{E_{f^\chi,p}}\le 2$, the assertion in \eqref{cbound1} is clear.
As for \eqref{cbound2},
since $\chi$ is a character mod $p^{e-k}$, it follows from
\cite[Theorem~3.1]{AL} that
$$
\ord_p\cond(f^\chi)\le e-k+\max(e-k,\ord_p{N}),
$$
so \eqref{cbound2} holds when $\deg{E_{f^\chi,p}}=0$.
In particular, this is the case when the local constituent $\pi_p$
of the automorphic representation associated to $f$ is
supercuspidal. If $\pi_p$ is special then we might have
$\deg{E_{f^\chi,p}}=1$, but this happens only when
$\ord_p\cond(f^\chi)=1$, in which case the left-hand side of
\eqref{cbound2} is at most $k+2\le e+1\le 2e$.
Finally, suppose that $\pi_p$ is a principal series, say
$\pi_p=\pi(\chi_1,\chi_2)$. If $\deg{E_{f^\chi,p}}=2$ then
$\ord_p\cond(f^\chi)=0$, so again we get the upper bound
$k+2\le2e$. If $\deg{E_{f^\chi,p}}=1$ then we may assume that
$\chi\chi_1$ is unramified, so that
$$
\ord_p\cond(f^\chi)=\ord_p\cond(\chi\chi_2)
\le\ord_p\cond(\chi_1)+\ord_p\cond(\chi_2)=\ord_p{N},
$$
and the left-hand side of \eqref{cbound2} is at most
$e+\ord_p{N}$.
\end{proof}

\begin{lemma} \label{Voronoi}
Let $a/q\in\Q$, and let $F:(0,\infty)\to\C$ be a $C^2$ function of
compact support. Define
$$
N_1=(N,q),\quad N_2=\frac{N}{N_1},\quad
q_2=(N_2^\infty,q),\quad
q_1=\frac{q}{q_2},
$$
and write
$$
\frac{a}{q}=\frac{a_1}{q_1}+\frac{a_2}{q_2},
$$
with the fractions on the right-hand side in lowest terms.
Then
\begin{equation}\label{e:voronoi}
\begin{aligned}
\sum_{n=1}^\infty\lambda_f(n) e\!\left(\frac{an}q\right) \! F(n)
=\sum_{\chi\pmod*{N_1}}\sum_{r\mid N_2q_2^2}
&c(f,a/q,r,\chi)\sum_{n=1}^\infty\lambda_{\bar{f}^\chi}(n)
e\!\left(-\frac{\overline{a_1r}n}{q_1}\right)\\
&\cdot2\pi i^k\int_0^\infty F(q_1x)
J_{k-1}\!\left(4\pi\sqrt{\frac{nx}{q_1r}}\right) \mathrm{d}x,
\end{aligned}
\end{equation}
where $\overline{a_1r}$ denotes a multiplicative inverse of
$a_1r\pmod*{q_1}$, and $c(f,a/q,r,\chi)\ll_{q_2}1$.
\end{lemma}
\begin{remark}
Note that $q_2\mid N_1$, so that both $r$ and the coefficients
$c(f,a/q,r,\chi)$ are $O_N(1)$, independent of $q$.
\end{remark}
\begin{proof}
Since $q_2\mid N_1$, we have $\lcm(Nq_2,q_2^2)=Nq_2$, so applying
Lemma~\ref{additivetwist} with $\alpha=a_2/q_2$, we get
$$
\lambda_f(n)e\!\left(\frac{a_2n}{q_2}\right)
=\sum_{\chi\pmod*{q_2}}
\sum_{m\big|\bigl(\frac{Nq_2}{\cond(f^\chi)},n\bigr)}
C(f,a_2/q_2,m,\chi)\lambda_{f^\chi}\!\left(\frac{n}{m}\right),
$$
for some numbers $C(f,a_2/q_2,m,\chi)\in\C$.

Next we compute the sum
$$
\sum_{\substack{n\ge1\\m\mid n}}
\lambda_{f^\chi}\!\left(\frac{n}{m}\right)
e\!\left(\frac{a_1n}{q_1}\right)F(n)
=\sum_{n=1}^\infty\lambda_{f^\chi}(n)e\!\left(\frac{a_1mn}{q_1}\right)F(mn)
$$
by applying the Voronoi summation formula \cite[Theorem~A.4]{KMV}.
Let us first suppose that $F$ is smooth, which is a hypothesis of
loc.~cit. Put $g=f^\chi$ and $D=\cond(g)$, so that $g\in
S_k(\Gamma_0(D),\psi)$, where $\psi=\xi\chi^2$.
Set $D_1=(D,q_1)=N_1/q_2$, $D_2=D/D_1$, and split the
character $\psi$ as a product $\psi_{D_1}\psi_{D_2}$ of characters
modulo $D_1$ and $D_2$, respectively. Then \cite[Theorem~A.4]{KMV}
yields
\begin{align*}
&\sum_{n=1}^\infty\lambda_g(n)e\!\left(\frac{a_1mn}{q_1}\right)F(mn)\\
&=\overline{\psi_{D_1}(a_1m)}\psi_{D_2}(-q_1)\frac{\eta_g(D_2)}{\sqrt{D_2}}
\sum_{n=1}^\infty\lambda_{g^{\overline{\psi}_{D_2}}}(n)
e\!\left(-\frac{\overline{a_1mD_2}}{q_1}\right)
\int_0^\infty F(mq_1x)J_{k-1}\!\left(4\pi\sqrt{\frac{nx}{q_1D_2}}\right) \mathrm{d}x,
\end{align*}
where $\eta_g(D_2)$ is a constant of modulus $1$ and $\overline{a_1mD_2}$
is an inverse of $a_1mD_2\pmod*{q_1}$.
Note that $g^{\overline{\psi}_{D_2}}=\bar{g}^{\psi_{D_1}}$, where
$\bar{g}\in S_k(\Gamma_0(D),\overline{\psi})$ is the dual of $g$.
Since $D_1$ is coprime to the modulus of $\chi$, we further have
$\psi_{D_1}=\xi_{D_1}=\xi_{N_1/q_2}$, so that
$\bar{g}^{\psi_{D_1}}=\bar{f}^{\overline{\chi}\xi_{N_1/q_2}}$.
Since $m$ is restricted to the divisors of $Nq_2/D$, we see that $mD_2$
divides $Nq_2/D_1=N_2q_2^2$. Writing $r=mD_2$ and
making the change of variables $x\mapsto x/m$,
the last line becomes
$$
\overline{\psi_{D_1}(a_1m)}\psi_{D_2}(-q_1)\frac{\eta_g(D_2)}{m\sqrt{D_2}}
\sum_{n=1}^\infty\lambda_{\bar{f}^{\overline{\chi}\xi_{N_1/q_2}}}(n)
e\!\left(-\frac{\overline{a_1r}n}{q_1}\right)
\int_0^\infty F(q_1x)J_{k-1}\!\left(4\pi\sqrt{\frac{nx}{q_1r}}\right) \mathrm{d}x.
$$
From the estimate provided by Lemma~\ref{additivetwist}, we have
$$
\overline{\psi_{D_1}(a_1m)}\psi_{D_2}(-q_1)\frac{\eta_g(D_2)}{m\sqrt{D_2}}
C(f,a_2/q_2,m,\chi)\ll_{q_2}1.
$$
Making the change of variables $\chi\mapsto\xi_{N_1/q_2}\overline{\chi}$
yields \eqref{e:voronoi}.

It remains only to see that \eqref{e:voronoi} is valid if $F$ is merely
$C^2$ and not necessarily smooth.
Making the substitution $x=q_1r(\frac{u}{4\pi})^2$, we have
\[
\int_0^\infty F(q_1x) J_{k-1}\!\left(4\pi\sqrt{\frac{nx}{q_1r}}\right) \mathrm{d}x
=\frac{q_1r}{8\pi^2} \int_0^\infty u^{-k}
F\!\left(\frac{q_1^2ru^2}{16\pi^2}\right) u^k J_{k-1}(\sqrt{n}u)\,\mathrm{d}u.
\]
Applying integration by parts twice and using the estimates
\[
\frac{\mathrm{d}}{ \mathrm{d}x} \left\{x^k J_k(x)\right\}= x^k J_{k-1}(x),
\quad
\frac{\mathrm{d}}{ \mathrm{d}x} \left\{x^{k+1} J_{k+1}(x)\right\}= x^{k+1} J_k(x),
\quad\text{and}\quad
J_{k+1}(x)\ll_k\frac1{\sqrt{x}},
\]
we see that this integral is $O(n^{-\frac54})$. Thus,
the sum over $n$ on the right-hand side of \eqref{e:voronoi}
is absolutely convergent, and the lemma follows by a standard
argument using smooth approximations of $F$.
\end{proof}

\section{Proof of Theorem~\ref{mainthm}}\label{s:proof}
\subsection{Initial reduction}
Let $f$ be as in the statement of Theorem~\ref{mainthm}, and
let $t\in\R$. By replacing $f$ with $\bar{f}$ if necessary,
we may assume without loss of generality that $t\ge0$.
We may further assume that
\begin{equation}\label{tbound}
t\ge\max\big\{k^{\frac32}\log{k},N^{\frac32},t_0\big\}
\end{equation}
for a large constant $t_0$, as otherwise the convexity bound implies
Theorem~\ref{mainthm}.

Let $C$ denote the analytic conductor defined in \eqref{conductor},
and fix, for the remainder of the paper, a choice of
$\delta\in\{0,1-\frac8{3\pi}\}$. With $\delta=0$, all implied constants
depend at most polynomially on $k$ and $N$.
In this section, we prove that
\begin{equation} \label{reduction1}
\begin{split}
|L(\tfrac{1}{2}+it,f)| &\ll
\sum_{M}  \frac{1}{\sqrt{M}} \max_{M_1 \in [M,2M]} \left|\sum_{M_1 \le n \le 2M} \lambda_f(n) \, n^{-it} \right|
+ O_{k,N}\Big(\sqrt{M_0} \, (\log M_0)^{-\delta} \Big)
\end{split}
\end{equation}
for any integer $M_0\in[2,\sqrt{C}]$, where $M$ runs through numbers
of the form $2^KM_0$ for integers $K\in[0,\log_2\frac{\sqrt{C}}{M_0}+1]$.
Therefore, in order to prove Theorem~\ref{mainthm}, it suffices to estimate exponential sums of the form
\begin{equation} \label{dyadicsum}
\sum_{M_1 \le n \le M_2} \lambda_f(n) \, n^{-it} 
\end{equation}
where $M_1\le M_2 \le 2 M_1$.  

Our starting point for the proof of \eqref{reduction1} is the
approximate functional equation for $L(\frac{1}{2}+it,f)$ in the form
of Lemma~\ref{Harcos}. We remark that, without loss of generality, we
may suppose that the test function $g$ appearing there is decreasing
and supported on the interval $[0,2]$. For example, the function
\[
g(x) = \left\{ \begin{array}{ll}
 1, &\mbox{ if $x < \frac{1}{2}$,} \\
  \displaystyle{ \alpha \int_{\log_2x}^{1} e^{-\frac{1}{1-t^2}} \, \mathrm{d}t}, &\mbox{ if $\frac{1}{2}\le x \le 2$,} \\
  0, &\mbox{ if $x>2$,}
       \end{array} \right.
\]
where $\alpha=e^{\frac12}/(K_1(\frac12)-K_0(\frac12))= 2.25228\ldots$ is chosen so that $g(\frac{1}{2})=1$, has these properties
($K_n(z)$ denotes  the usual $K$-Bessel function). Since 
$|\epsilon_f\Gamma_\C(\tfrac{k}{2}\!-\!it)
/\Gamma_\C(\tfrac{k}{2}\!+\!it)|=1$,
by \eqref{tbound} we have
\[
|L(\tfrac{1}{2}+it,f)| \le 2 \left| \sum_{n =1}^\infty
\frac{\lambda_f(n)}{n^{\frac12+it}} \, g\!\left(\frac{n}{\sqrt{C}} \right) \right| + O\!\left( 1 \right).
\]

By the triangle inequality, since $0\le g(x) \le 1$, we have
\begin{equation}\label{triangle}
\begin{split}
\Bigg| \sum_{n =1}^\infty &\frac{\lambda_f(n)}{n^{\frac12+it}}   \,
g\!\left(\frac{n}{\sqrt{C}} \right) \Bigg| \le \sum_{n\le M_0}
\frac{|\lambda_f(n)|}{\sqrt{n}} + \sum_{M} \left|\sum_{M < n \le 2M}
\frac{\lambda_f(n)}{n^{\frac12+it}}\, g\!\left(\frac{n}{\sqrt{C}}
\right) \right|,
\end{split}
\end{equation}
with $M_0\le\sqrt{C}$ and $M=2^KM_0$ as above.  Note that only finitely
many values of $M$ are relevant, since $\supp(g)\subseteq [0,2]$. We
will choose $M_0$ (depending on $t$) at the end of the proof, but to
fix ideas, we note that
\begin{equation}\label{Mbound}
\left(\frac{t}{\log{t}}\right)^{\frac23}\ll M_0 \ll t^{\frac23}.
\end{equation}
Applying Lemma~\ref{l:L1average}(iv), we have
\[
\sum_{n\le M_0} \frac{|\lambda_f(n)|}{\sqrt{n}} \ll_{k,N}
\sqrt{M_0}(\log M_0)^{-\delta}.
\]

We now simplify the second sum on the right-hand side of \eqref{triangle}
using \cite[Lemma~5.1.1]{Huxley}.
Defining $G(x)=\frac1{\sqrt{x}}g(\frac{x}{\sqrt{C}})$, we may assume that $G$ is decreasing, and hence 
\[
\left|\sum_{M < n \le 2M} \frac{\lambda_f(n)}{n^{\frac12+it}}\, g\!\left(\frac{n}{\sqrt{C}} \right) \right| \ll \frac{1}{\sqrt{M}} \max_{M_1 \in [M,2M]} \left|\sum_{M_1 \le n \le 2M} \lambda_f(n) \, n^{-it} \right|.
\]
Combining estimates yields  \eqref{reduction1}.

\subsection{Farey fractions and stationary phase}\label{s:Farey}
We now turn our attention to estimating the sum in \eqref{dyadicsum}. Let $M$ be a size parameter and suppose that $M_0\le M \ll M_1 \le M_2 \le 2 M_1 \ll M$. Define
\[
R=\sqrt{\frac{M}{M_0}} \quad \text{and} \quad H := \left\lceil \frac{M^2}{R^2t} \right\rceil,
\]
and let
\[
\mathcal{F}(R) = \left\{ \frac{u}{v} : u,v\in\Z, (u,v) = 1, 0 < v \le R \right\}
\] 
denote the set of extended Farey fractions with denominator less than or equal to $R$. Consider the interval 
\[
 \Big[  -\frac{t}{2 \pi (M_1+2H)}, -\frac{t}{2 \pi (M_2-2H)} \Big]
\]
and denote the elements of $\mathcal{F}(R)$ in this interval by
$\alpha_j$, $j=1,\ldots, J$, in increasing order. (If $M_2-M_1 <
4H$ or if this interval contains no elements of $\mathcal{F}(R)$
then \eqref{dyadicsum} is trivially bounded by the error term in
\eqref{shiu}, below. Hence we may assume that this is not the case.) We
make the labeling $\alpha_j = -\frac{u_j}{v_j}$, where $u_j, v_j \in
\Z_{>0}$, $(u_j, v_j)=1$, and $v_j \le R$.

For consecutive Farey fractions
$\alpha_j = -\frac{u_j}{v_j}$ and  $\alpha_{j+1} = -\frac{u_{j+1}}{v_{j+1}}$
the mediant, denoted $\rho_j$, is given by 
$\rho_j = -\frac{u_j+u_{j+1}}{v_{j}+v_{j+1}}$. 
Note that 
\begin{equation}
 \label{fareydiff}
|\rho_j-\alpha_j|=\frac{1}{v_j(v_j+v_{j+1})} \asymp\frac{1}{v_j R}
\end{equation}
and similarly $|\rho_j-\alpha_{j+1}| \asymp\frac{1}{v_j R}$.
Define the function $h(y) = -\frac{t}{2 \pi y}$
and integers $\mathcal{N}_0 = M_1+2H,$ $\mathcal{N}_J = M_2 -2H,$
and $\mathcal{N}_j=\lfloor{h(\rho_j)+\frac12}\rfloor$
for $j=1,\ldots,J-1$. Evidently,
\[
   M_1 <\mathcal{N}_0 < \mathcal{N}_1 < \cdots < \mathcal{N}_{J-1} < \mathcal{N}_J < M_2.
\]
Using the above and assuming that $t_0$ is sufficiently large, we have
\[
\begin{split}
  \mathcal{N}_j - h(\alpha_j) &= h(\rho_j) -h(\alpha_j) + O(1) 
  = \frac{t}{2 \pi} \frac{\rho_j- \alpha_j}{ \alpha_j \rho_j}+ O(1) 
  \\
  &\asymp \frac{t}{ (\frac{t}{M})^2 v_j R }
  \asymp \frac{HR}{v_j}. 
\end{split}
\]
By a similar argument we see that
$h(\alpha_j)-\mathcal{N}_{j-1} \asymp \frac{HR}{v_j}$, and thus
\begin{equation}
  \label{omegajsupport}
\mathcal{N}_{j}-\mathcal{N}_{j-1} \asymp \frac{HR}{v_j}.
\end{equation} 
Next let 
\begin{equation}
  \label{omegaj}
\omega_j(x)=\omega(x-\mathcal{N}_{j-1}) -\omega(x-\mathcal{N}_j),
\end{equation}
where
\begin{equation*}
 \omega(x) =  \begin{cases}
 1, & \text{ for } x \ge H, \\
 \frac{1}{2}(1+\sin^{s+1}(\frac{\pi x}{2H})), & \text{ for } |x| \le H, \\
 0, & \text{ for } x \le -H,
 \end{cases}
\end{equation*}
for an integer $s\ge2$. 
These functions provide a $C^s$ partition of unity on the interval $[M_1+2H, M_2-2H]$.  
In particular,  
\begin{equation*} 
  \sum_{j=1}^{J} \omega_j(n) = 
  \begin{cases}
   0, & \text{ for } x \le M_1, \\
   1, & \text{ for } M_1+2H \le x \le M_2-2H, \\
   0, & \text{ for } x \ge M_2. 
  \end{cases}
\end{equation*}
From this we observe that
\begin{equation}\label{shiu}
\begin{split}
  \Bigg| \sum_{M_1 \le n \le M_2} \frac{\lambda_f(n)}{n^{it}} -
  \sum_{j=1}^{J} \sum_{n =1}^{\infty} \lambda_f(n)n^{-it} \omega_j(n) \Bigg|
  &\le \sum_{m=M_1}^{M_1+2H} |\lambda_f(m)| + \sum_{m =M_2-2H}^{M_2} |\lambda_f(m)|.
  \\
  & \ll_{k,N} \max(H,M^{\frac35})(\log M)^{-\delta},
\end{split} 
\end{equation}
by Lemma~\ref{l:L1average}(ii). Hence
\begin{equation}  \label{difference}\
\begin{split}
  \sum_{M_1 \le n \le M_2} \frac{\lambda_f(n)}{n^{it}} &=
  \sum_{j=1}^{J} \sum_{n =1}^{\infty} \lambda_f(n)n^{-it} \omega_j(n)  +
  O_{k,N}(\max(H,M^{\frac35})(\log M)^{-\delta} )
  \\
  &=  \sum_{j=1}^{J} \sum_{n =1}^{\infty} \lambda_f(n)  e(\alpha_j n) F_j(n) + 
  O_{k,N}(\max(H,M^{\frac35})(\log M)^{-\delta} ),
\end{split}
\end{equation}
where $F_j(n) = n^{-it} e(-\alpha_j n) \omega_j(n)$. 

We would now want to apply Lemma \ref{Voronoi} to the sum involving $F_j(n)$. 
If the level $N$ is not squarefree, then we do not have a Voronoi
formula for every Farey fraction $\alpha_j$. To circumvent this issue,
we decompose each fraction into a `good part' and a `bad part' part where
the bad parts range over a finite set, the additive twists involving
the good part of the fraction can be handled using Voronoi summation,
and the additive characters involving the bad part can be handled by
decomposing into multiplicative characters. To that end, define
\begin{equation*}
  N^{\flat}  = N \prod_{p \mid N} p^{-1} \quad \text{ and } \quad
	\mathcal{B}(N^\flat)  = \left\{ \frac{b}{N^{\flat}} : 0 \le b < N^{\flat}\right\}
 \end{equation*}
and write
\begin{equation*}
   \alpha_j = - \frac{u_j}{v_j} = -\frac{a_j}{q_j} + \beta_j = -\frac{a_j}{q_j} + \frac{c_j}{d_j} 
\end{equation*}
where $a_j,q_j,c_j,d_j \in \Z_{\ge 0}$,
$(a_j,q_j)=(c_j,d_j)=(q_j,d_j)=1$, $\beta_j \in \mathcal{B}(N^\flat)$, and for
every prime $p\mid q_j$ we have $\ord_p(q_j) \ge
\ord_p(N)$. Such a  decomposition always exists and is uniquely
determined; concretely,
\[
d_j = \prod_{\substack{p\mid v_j \\ \ord_p(v_j)<\ord_p(N)}} p^{\ord_p(v_j)}, \quad q_j =\frac{v_j}{d_j},
\]
and $c_j$ is the unique integer in $[0,d_j)$ satisfying $
q_j c_j  \equiv -u_j\pmod*{d_j}$.   Since  $v_j=d_jq_j$, this congruence is equivalent to 
\begin{equation*}
    d_ju_j \equiv -c_jv_j\pmod*{d_{j}^2}.
\end{equation*}

Next we apply Voronoi summation in the form of Lemma \ref{Voronoi} to see that
\begin{equation}\label{applyVoronoi}
\begin{split}
  \sum_{n=1}^{\infty} &\lambda_f(n)  e(\alpha_j n) F_j(n)  = 
  \\
  & \sum_{\beta \in \mathcal{B}(N^\flat)} \sum_{\substack{r \mid N N^{\flat} \\
	(r,q_j)=1}} \sum_{\chi\pmod*{N}} 
  c(f,r,\chi;j) 
   \sum_{\ell=1}^{\infty}  \lambda_{\bar{f}^{\chi}}(\ell) e( \tfrac{ \overline{ r a_j} \ell}{q_j}) 
   2 \pi i^k \int_{0}^{\infty} F_j(q_jx) J_{k-1}\Big(4 \pi \sqrt{ \tfrac{\ell x}{r q_j}}\Big) \, \mathrm{d}x
\end{split}
\end{equation}
for some complex numbers $c(f,r,\chi;j)$ satisfying $c(f,r,\chi;j)
\ll_{N^\flat} 1$. Applying stationary phase to the integral on the
right-hand side of this equation we derive the following proposition,
deferring the proof until Section~\ref{s:statphase}.
\begin{proposition} \label{NathanProp}
Given $\beta=\frac{c}{d}\in \mathcal{B}(N^\flat)$ and $r\mid NN^\flat$,
let
$$
J(\beta,r) = \{j\in\{1,\ldots,J\}:\beta_j =\beta, (q_j,r)=1 \}.
$$
If $s \ge 6$, then  
\begin{equation}
\begin{split}
  \label{keyidentity}
 \sum_{j \in J(\beta,r)} & c(f,r,\chi;j)  \sum_{\ell=1}^{\infty}  \lambda_{\bar{f}^{\chi}}(\ell) e( \tfrac{ \overline{ r a_j} \ell}{q_j}) 
   2 \pi i^k \int_{0}^{\infty} F_j(q_jx) J_{k-1}\Big(4 \pi \sqrt{ \tfrac{\ell x}{r q_j}}\Big) \, \mathrm{d}x
   \\
   &= \sum_{j \in J(\beta,r)} c(f,r,\chi;j) 
    \sum_{ \pm } (\mp1)^k
   \sum_{\ell\le K_1rd^{-2}}
	 \lambda_{\bar{f}^{\chi}}(\ell) e( \tfrac{ \overline{ r a_j} \ell}{q_j})  
    \omega_j(x_{j}^{\pm}(\tfrac{\ell}{r})) h_{j}^{\pm}(\tfrac{\ell}{r})  e(g_{j}^{\pm}(\tfrac{\ell}{r})) 
    \\
& \quad  +O_{k,N,s}\!\left(\Big( \sqrt{M} \Big( \frac{M}{R^2} \Big)^{\frac{1}{2(s-1)}}
 + \frac{M^{\frac{5}{2}} R^2}{H^3}  \Big) (\log M_0)^{-\delta} \right),
\end{split}
\end{equation}
 where $x_{j}^{\pm}(\ell)$ are stationary points defined by 
\begin{equation}
  \label{xjlpm}
  x_{j}^{\pm}(\ell) = \left(\frac{d}{2u_j}\right)^2
\left(\sqrt{\ell+\frac{2tu_jq_j}{\pi d}}\pm \sqrt{\ell}\right)^2,
\end{equation}
\begin{equation}
\begin{split}
\label{gjlpm}
 g_{j}^{\pm}(\ell)   &= - \frac{t}{2 \pi} \log  x_{j}^{\pm}(\ell) + \frac{u_j}{v_j} x_j^\pm(\ell) \mp \frac{2}{q_j} \sqrt{\ell  x_j^\pm(\ell)} + \frac{1}{8} \mp \frac{1}{8},
\end{split}
  \end{equation}
\begin{equation}
  \label{hjlpm}
   h_{j}^{\pm}(\ell)  = \left( \frac{q_j t \sqrt{\ell}}{\pi
	 x_{j}^{\pm}(\ell)^{\frac32}} \pm \frac{\ell}{x_{j}^{\pm}(\ell)}
	 \right)^{-\frac12},
\end{equation}
and $K_1\asymp M/R^2=M_0$.
\end{proposition}

Combining \eqref{difference}, \eqref{applyVoronoi}, and Proposition \ref{NathanProp}, we have
\begin{equation}
\begin{split}
  \label{truncatedvoronoi}
     &\sum_{M_1 \le n \le M_2} \frac{\lambda_f(n)}{n^{it}} = \\
  &  \sum_{\beta \in \mathcal{B}(N^\flat)} \sum_{r \mid N N^{\flat}}
	\sum_{\chi\pmod*{N}} 
   \sum_{j \in J(\beta,r)} c(f,r,\chi;j) 
    \sum_{ \pm } (\mp1)^k
   \sum_{\ell\le K_1rd^{-2}}
	 \lambda_{\bar{f}^{\chi}}(\ell) e( \tfrac{ \overline{ r a_j} \ell}{q_j})  
    \omega_j(x_{j}^{\pm}(\tfrac{\ell}{r})) h_{j}^{\pm}(\tfrac{\ell}{r})  e(g_{j}^{\pm}(\tfrac{\ell}{r})) \\
&  \qquad +O_{k,N,s}\!\left(\Big( \sqrt{M} \Big( \frac{M}{R^2} \Big)^{\frac{1}{2(s-1)}}
 + \frac{M^{\frac{5}{2}} R^2}{H^3} +\max(H,M^{\frac35})  \Big) (\log M_0)^{-\delta} \right).
\end{split}
\end{equation}
Since $R=\sqrt{M/M_0}\ge1$, $M\ll\sqrt{N}t$, $M_0 \ll t^{\frac23}$,
and $|t|\ge N^{\frac32}$, we see that
$$
\max(H,M^{\frac35}) \ll  \frac{M^{\frac{5}{2}} R^2}{H^3}
\quad\text{and}\quad
\frac{M}{R}\ll\frac{R^8t^3}{M^{\frac72}}.
$$
Therefore, assuming $s\ge 2$ and using the definition of $H$,
we find that the error term in \eqref{truncatedvoronoi} can be replaced by
$O_{k,N}(R^8t^3M^{-\frac72}(\log M_0)^{-\delta})$.

The next step is to split the sum in \eqref{truncatedvoronoi} so that the integers $\ell$ lie in dyadic intervals and 
 the sum over $j$ is reorganized so that the $v_j$ lie in dyadic intervals.   Thus we have
 \begin{equation}
\begin{split}
  \label{dyadic}
     \sum_{M_1 \le n \le M_2} \frac{\lambda_f(n)}{n^{it}} & =   \sum_{\beta \in \mathcal{B}(N^\flat)} \sum_{r \mid N N^{\flat}}  
	\sum_{\chi\pmod*{N}}   \sum_{ \pm } (\mp1)^k  \sum_{L,U,V}  \\
  & \cdot  \sum_{\ell=L_1}^{L_2}
	 \lambda_{\bar{f}^{\chi}}(\ell)
   \sum_{\substack{ j \in J(\beta,r) \\  (u_j,v_j)=1  \\
   U_1 \le u_j \le U_2,  V_1 \le v_j \le V_2 \\  du_j \equiv -cv_j\pmod*{d^2}. 
    }} c(f,r,\chi;j) 
    e( \tfrac{ \overline{ r a_j} \ell}{q_j})  
    \omega_j(x_{j}^{\pm}(\tfrac{\ell}{r})) h_{j}^{\pm}(\tfrac{\ell}{r})  e(g_{j}^{\pm}(\tfrac{\ell}{r})) \\
    &  \qquad +O_{k,N}(R^8t^3M^{-\frac72}(\log M_0)^{-\delta})
\end{split}
\end{equation}
where
$L_1, L_2, U_1, U_2, V_1,$ and $V_2$ are positive integers satisfying $L \ll L_1 \le L_2 \le L$, $U \ll U_1 \le U_2 \le U$, and $V \ll V_1 \le V_2 \le V$ and $L, U,V$ run through powers of 2 satisfying
\[
  L\ll \frac{rM}{(dR)^2}, \quad V\ll R, \quad \text{and} \quad U \asymp \frac{tV}{M}.
\]
Observe that the third condition follows from the fact that if $v_j \asymp V$, then $u_j \asymp \frac{tV}{M}$ since
$\frac{u_j}{v_j} \asymp \frac{t}{M}$.   The next step 
is to apply a large sieve inequality for Farey fractions to the inner double sum in \eqref{dyadic}.  In fact, we bound a
more general sum with Proposition \ref{Large Sieve}, below.  

\subsection{The large sieve and conclusion of the proof}
To estimate the main term on the right-hand side of
\eqref{truncatedvoronoi}, we employ the following large sieve
inequality, deferring the proof until Section~\ref{s:largesieve}.

\begin{proposition}  \label{Large Sieve} 
Let notation be as above and fix $\beta=\frac{c}{d}\in \mathcal{B}(N^\flat)$ and $r\mid NN^\flat$.
Let $L_1, L_2, U_1, U_2, V_1,$ and $V_2$ be positive integers satisfying $L \ll L_1 \le L_2 \le L$, $U \ll U_1 \le U_2 \le U$, and $V \ll V_1 \le V_2 \le V$ where $L,U,$ and $V$ are size parameters satisfying
\[
L\ll \frac{rM}{(dR)^2}, \quad V\ll R, \quad \text{and} \quad U \asymp \frac{tV}{M}.
\]
Define
\[
\mathcal{R} = \Big\{ (u,v) \in \Z^2 :
U_1\le u \le U_2,\, V_1\le v \le V_2,\,(u,v)=1,\,
du\equiv-cv\pmod*{d^2}\Big\}
\]
and
\[
\eta=\sqrt{\frac{d^2t}{rLUV}},\quad
X_0 = \sqrt{L\max(\eta,1)}.
\]
Then, given any complex numbers
$\nu(j)$, $\lambda(\ell)$ for
$j\in J(\beta,r)$ and $\ell\in\Z\cap[L_1,L_2]$,
we have 
\begin{equation}\label{large_sieve_estimate}
\begin{split}
\Bigg| \sum_{\ell=L_1}^{L_2}&\lambda(\ell)
\sum_{ \substack{j \in J(\beta,r) \\ (u_j,v_j) \in \mathcal{R}} }
\nu(j)h_{j}^\pm(\ell/r)
\omega_j\big( x_j^\pm(\ell/r) \big)  e\big( g_{j}^\pm(\ell/r) \big) \Bigg|^2
\\
&\ll  \max_{j\in J(\beta,r)}\big|\nu(j)\big|^2
\sum_{\ell=L_1}^{L_2}|\lambda(\ell)|^2
\cdot\frac{\eta rV}{U}
\left\{ X_0\big(\#\mathcal{R}\big)^2
+\int_{X_0}^L B\big(\Delta_1(X),\Delta_2(X)\big)\,\mathrm{d}X \right\},
\end{split}
\end{equation}
where the implied constant is absolute,
$\Delta_1$ and $\Delta_2$ are functions satisfying
\begin{equation}\label{DeltaConditions}
\Delta_1(X) \ll \frac{L}{X^2}
\quad\text{and}\quad
\Delta_2(X) \ll\frac{L}{\eta X^2},
\end{equation}
and $B(\Delta_1,\Delta_2)$ is the number of pairs
$(i,j)\in J(\beta,r)^2$ such that $(u_i,v_i),(u_j,v_j)\in\mathcal{R}$,
\begin{equation} \label{fraction_condtions}
\Bigg\|\frac{\overline{a_i r}}{q_i}-\frac{\overline{a_j r}}{q_j}\Bigg\|
\le\Delta_1,
\quad\text{and}\quad
|u_iv_i-u_jv_j|\le UV\Delta_2.
\end{equation}
\end{proposition}

For fixed $\beta=\frac{c}{d}$ and $r$, define 
$\varphi(y) = \sqrt{y-\beta r}$
and recall that 
$\frac{u_j}{v_j}=\frac{a_j}{q_j} - \beta$
for $j\in J(\beta,r)$. This, together with the condition $|u_iv_i-u_jv_j| \le UV\Delta_2$, implies that
\[
\frac{d}{\sqrt{r}}
\left|q_i \varphi\big(ra_i/q_i)-q_j \varphi\big(ra_j/q_j) \right|
= \left|\sqrt{u_iv_i} - \sqrt{u_j v_j}\right| \ll \sqrt{UV} \Delta_2.
\]
We want to count pairs of fractions
$\big(\frac{ra_i}{q_i},\frac{ra_j}{q_j}\big)$
in the interval
$\left[ \frac{rU_1}{V_2}+\beta r,\frac{rU_2}{V_1}+\beta r\right]$
satisfying the inequalities in \eqref{fraction_condtions}.
On this interval, we have
$\varphi(y) \asymp \Phi:=\sqrt{ \frac{rU}{V} }$,
and hence
\begin{equation*} 
\left| q_i \varphi\big(ra_i/q_i)-q_j \varphi\big(ra_j/q_j) \right|\ll\frac{\Phi V}{d}\Delta_2.
\end{equation*}
Note that under the conditions of Proposition \ref{Large Sieve}, we have
$$
\frac{V_1}{d}\le q_j\le\frac{V_2}{d}
\quad\mbox{and}\quad
\frac{r}{d}\left(U_1+\beta V_1\right)
\le ra_j\le
\frac{r}{d}\left(U_2+\beta V_2\right).
$$
Assuming that $U_2\le2U_1$ and $V_2\le 2V_1$, we thus
see that $ra_j$ and $q_j$ lie
in dyadic intervals. This allows us to apply the following estimate of Graham and Kolesnik \cite{GK} with parameters
$A=\frac{r}{d}\Big(U_1+\beta V_1 \Big)$
and
$C=\frac{V_1}{d}$.

\begin{lemma} \label{magicmatrices}
Suppose $A$ and $C$ are positive integers, and that $\Delta_1$ and $\Delta_2$ are positive real numbers not exceeding 1. Suppose that $\varphi(x)$ is a real positive continuously differentiable function defined on a subinterval $I$ of $[A/(2C),2A/C]$. Suppose that there is a constant $C_0$ and a parameter $\Phi$ such that
\[
\frac{\Phi}{C_0} \le \varphi(x) \le C_0\Phi, \quad \frac{\Phi}{C_0} \le |x\varphi'(x)| \le C_0\Phi, \quad \text{and} \quad \frac{\Phi}{C_0} \le |\varphi(x) \!-\! x\varphi'(x)| \le C_0\Phi
\]
whenever $x$ is an element of $I$. Let $B$ be the number of solutions of the inequalities 
\[
\Big\| \frac{\overline{r}}{q}-\frac{\overline{r_1}}{q_1}\Big\| \le \Delta_1 \quad \text{and} \quad \Big| q \varphi(r/q)-q_1 \varphi(r_1/q_1) \Big| \le C \Phi \Delta_2
\]
when $(r,q)=1, (r_1,q_1)=1, A < r, r_1 \le 2A,$ and $C<q,q_1 \le 2C$
where $r \overline{r} \equiv 1 \pmod*{q}$ and $r_1 \overline{r_1} \equiv 1
\pmod*{q_1}$ . Then
\[
B \ll \Delta_1 \Delta_2 A^2 C^2 + \Delta_1^2 A^2 C^2 + AC + \Delta_2 A^2 + \Delta_2 C^2
\]
where the implied constant depends only on $C_0$. 
\end{lemma}

\begin{proof}
This is  \cite[Lemma 7.18]{GK} and is a variation of a counting problem
first considered by Bombieri and Iwaniec \cite{BI} in connection to
subconvexity estimates for the Riemann zeta-function.
\end{proof}

Since $V\ll R$, $U\ll\frac{tR}{M}$, and $L\ll rM/(dR)^2$, we have
$LUV\ll rt/d^2$. Hence
\begin{equation}\label{Xnaught}
\eta \asymp_N \sqrt{\frac{t}{LUV}} \gg_N 1 \quad \text{and}\quad X_0
\asymp_N \sqrt{\eta L} \asymp_N \left(\frac{tL}{UV}\right)^{\frac14}.
\end{equation}
Therefore, in the notation of Proposition \ref{Large Sieve},  we have
\[
\begin{split}
X_0\big(\#\mathcal{R}\big)^2 &+ \int_{X_0}^L B\big(\Delta_1(X),\Delta_2(X)\big)\,\mathrm{d}X
\\
& \ll X_0\big(\#\mathcal{R}\big)^2 + \int_{X_0}^L \left\{ A^2C^2+\frac{L^2}{X^4}(1+\eta^{-1})+AC+\frac{L}{\eta X^2}(A^2+C^2) \right\} \mathrm{d}X
\\
& \ll X_0\big(\#\mathcal{R}\big)^2+ \frac{A^2C^2L^2}{X_0^3}(1+\eta^{-1})+ACL+\frac{L}{\eta X_0}(A^2+C^2)
\\
& \ll_N X_0 U^2 V^2 + \frac{U^2V^2L^2}{X_0^3}+LUV+\frac{U^2L}{\eta X_0^2}.
\end{split}
\]
The inequality $\eta \gg_N 1$ implies that $X_0 \gg_N \sqrt{L}$ which, in turn, implies that the first term on the right-hand side of the above expression dominates the second and fourth terms. Thus, using the estimates in \eqref{Xnaught} and that $U\asymp \frac{tV}{M}$, we have
\[
\begin{split}
&\frac{\eta rV}{U}
\left\{ X_0\big(\#\mathcal{R}\big)^2
+\int_{X_0}^L B\big(\Delta_1(X),\Delta_2(X)\big)\,\mathrm{d}X \right\}
\ll_N  \left(LM V^2 \right)^{\frac12} + \left( L^{-1} M^3 V^2
\right)^{\frac14}. 
\end{split}
\]
We apply
the large sieve with $\lambda(\ell)=\lambda_{\bar{f}^\chi}(\ell)$
and $\nu(j)=c(f,r,\chi;j)$. By Lemma~\ref{l:L1average}(i), we have
$\sum_{L_1<\ell \le L_2} |\lambda(\ell)|^2 \ll_{k,N} L$,
and thus the left-hand side of the large sieve inequality is
$\ll_{k,N} (L^3 M V^2)^{\frac12}+ (L^3M^3V^2)^{\frac14}$.
Estimating the sums over $\beta$, $r$, and $\chi$ trivially in \eqref{truncatedvoronoi}, we deduce that 
\[
\sum_{M_1 \le n \le M_2} \frac{\lambda_f(n)}{n^{it}}  \ll_{k,N}
\sum_{L,U,V} \Big\{ (L^3 M V^2)^{\frac14}+ (L^3M^3V^2)^{\frac18} \Big\}
+ \frac{R^8 \, t^3}{M^{\frac72}} (\log M_0)^{-\delta},
\]
where $L$, $U$, and $V$ run over powers of 2 satisfying $L\ll rM/(dR)^2$,
$V\ll R$, and $U\asymp\frac{tV}{M}.$ Note that there are boundedly many
values of $U$ corresponding to each $V$. Hence, we derive that
\[
\sum_{M_1 \le n \le M_2} \frac{\lambda_f(n)}{n^{it}}
\ll_{k,N} \frac{M}{R} +  \frac{R^8 \, t^3}{M^{\frac72}} (\log M_0)^{-\delta}
\ll\sqrt{M} \left\{\sqrt{M_0}+M_0^{-4}t^3(\log t)^{-\delta}\right\}.
\]
Therefore \eqref{reduction1} becomes
\[
\begin{split}
|L(\tfrac12+it,f)| &\ll_{k,N} \left\{ \sqrt{M_0}+ M_0^{-4} t^3 (\log t)^{-\delta}  \right\} \log C + \sqrt{M_0} (\log M_0)^{-\delta}
\\
&\ll_{k,N}  \sqrt{M_0} \log t + M_0^{-4} t^3 (\log t)^{1-\delta}.
\end{split}
\]
Choosing $M_0=\lceil t^{\frac23}(\log t)^{-\frac{2\delta}{9}}\rceil$
balances the two terms on the right-hand side and proves Theorem \ref{mainthm}.

\section{Proof of Proposition~\ref{NathanProp}}\label{s:statphase}

\subsection{Preliminary lemmas} Before proving Proposition~\ref{NathanProp}, we state three lemmas. 
For $k$ a positive integer, let $C^k([\alpha,\beta])$ denote
the space of $k$ times continuously differentiable real-valued functions on the interval $[\alpha,\beta]$.
The next two lemmas on exponential integrals are Lemma 5.5.5 and Lemma 5.5.6 of \cite{Huxley}.
\begin{lemma}  \label{weightedfirstderiv}
Let $F \in C^3([\alpha,\beta])$ and let $G \in C^2([\alpha,\beta])$.
Suppose there exist positive parameters $M,H,t,U$, with $M \ge \beta-\alpha$,
and positive constants $C_{r_1},C_{r_2}$ such that, for $x \in [\alpha,\beta]$, we have
\[
|F^{(r_1)}(x)| \le C_{r_1} t/M^{r_1} \quad \text{and} \quad |G^{(r_2)}(x)| \le C_{r_2} U/H^{r_2}
\]
for $r_1\in\{2,3\}$, and $r_2\in\{0,1,2\}$.  If $F'$ and $F''$ do not change sign on $[\alpha,\beta]$,
then 
\begin{equation*}
\begin{split}
 I=\int_{\alpha}^{\beta} G(x) e(F(x)) \, \mathrm{d}x & 
  = \frac{G(\beta) e(F(\beta))}{2 \pi i F'(\beta)} - \frac{G(\alpha) e(F(\alpha))}{2 \pi i F'(\alpha)} \\
  & \quad + O\!\left(  
  \frac{tU}{M^2} \Big( 1+ \frac{M}{H}+\frac{M^2}{H^2} \frac{\min |F'(x)|}{t/M} \Big) \frac{1}{\min |F'(x)|^3}
  \right).
\end{split}
\end{equation*}
\end{lemma}
\begin{lemma} \label{stationaryphase}
Let $F \in C^4([\alpha,\beta])$  
and let $G \in C^3([\alpha,\beta])$.
Suppose there exist positive parameters $M,H,T,U$, with $M \ge \beta-\alpha$,
$H \ge M/\sqrt{t}$,
and positive constants $C_{r_1},C_{r_2}$ such that, for $x \in [\alpha,\beta]$, we have
\[ |F^{(r_1)}(x)| \le C_{r_1} t/M^{r_1} \quad \text{and} \quad  |G^{(r_2)}(x)| \le C_{r_2} U/H^{r_2}
\]
for $r_1 \in \{2,3,4\}$ and $r_2\in\{0,1,2,3\}$, and a positive constant $\tilde{C}$ such that
\[
  F^{(2)}(x) \ge t/\tilde{C} M^2.
\]
Suppose also that $F'(x)$ changes sign from negative to positive at a point $x=\gamma$
with $\alpha < \gamma < \beta$.  If $t$ is sufficiently large with respect to the constants 
$C_{r_1},C_{r_2},\tilde{C}$, then
\begin{equation*}
\begin{split}
  \int_{\alpha}^{\beta} G(x) e(F(x)) \, \mathrm{d}x  &= \frac{G(\gamma) e(F(\gamma)+\frac{1}{8})}{\sqrt{F''(\beta)}} +
  \frac{G(\beta) e(F(\beta))}{2 \pi i F'(\beta)} - \frac{G(\alpha) e(F(\alpha))}{2 \pi i F'(\alpha)} \\
  &\quad + O\!\left( \frac{M^4U}{t^{2}}\Big( 1+ \frac{M}{H} \Big)^2
  \Big( \frac{1}{(\gamma\!-\!\alpha)^3} + \frac{1}{(\beta\!-\!\gamma)^3}  \Big)
   \right) + O\!\left( \frac{MU}{t^{\frac{3}{2}}} \Big( 1+ \frac{M}{H} \Big)^2 \right).
\end{split}
\end{equation*}
\end{lemma}

The third lemma provides bounds for derivatives of $F_j$. 
\begin{lemma} \label{Andylemma}
For any integer $s \ge 0$, we have
$$
\left|\frac{\mathrm{d}^s}{ \mathrm{d}x^s}\left\{F_j(x)\, x^{-\frac{k-1}2}\right\}\right|
\ll_s(v_jR)^{-s}x^{-\frac{k-1}2}.
$$
\end{lemma}
\begin{proof}
By continuity, it suffices to consider
$x=x_0\in(\mathcal{N}_{j-1}-H,\mathcal{N}_j+H)
\setminus\{\mathcal{N}_{j-1}+H,\mathcal{N}_j-H\}$.
For any fixed $x_0$, the function $F_j(z)z^{-\frac{k-1}2}$ agrees with an analytic function
$g(z)$ for $z$ in a neighborhood of $x_0$.
We estimate $g^{(s)}(x_0)=\frac{\mathrm{d}^s}{ \mathrm{d}x^s}(F_j(x)x^{-\frac{k-1}2})|_{x=x_0}$
via the Cauchy integral formula.
Since $v_jR\le R^2\ll M$, we may fix a small constant $c>0$ such that
$Y=cv_jR\le\frac12M_1$, and integrate over the circle $C_Y(x_0)$ of
radius $Y$ around $x_0$:
$$
|g^{(s)}(x_0)|=\left|\frac{s!}{2\pi i}
\int_{C_Y(x_0)}\frac{g(z)}{(z\!-\!x_0)^{s+1}}\,\mathrm{d}z\right|
\le s! \, Y^{-s} \!\! \sup_{z\in C_Y(x_0)}|g(z)|
\ll_s (v_jR)^{-s}\sup_{z\in C_Y(x_0)}|g(z)|.
$$
Hence, it suffices to show that $g(z)\ll_sx_0^{-\frac{k-1}2}$
for $z\in C_Y(x_0)$.

Let $z=x+iy\in C_Y(x_0)$, so that $x=x_0+O(v_jR)\asymp M$ and $y\ll v_jR$.
We have
$\frac{z}{x_0}-1\ll\frac{v_jR}{x_0}\ll\frac1{M_0}$
so, by \eqref{tbound} and \eqref{Mbound}, $z^{-\frac{k-1}2}=x_0^{-\frac{k-1}2}e^{O((k-1)/M_0)}\ll
x_0^{-\frac{k-1}2}$. Next, observe that
$$
|z^{-it}e(-\alpha_jz)|=e^{t\arctan(y/x)+2\pi\alpha_jy}.
$$
The exponent here is bounded since
$$
t\arctan\!\left(\frac{y}{x}\right)-\frac{ty}{x}
\ll\frac{t(v_jR)^3}{M^3}\ll\frac{t}{M_0^3}\ll 1.
$$
Moreover, by the estimates in Section~\ref{s:Farey}, we have
$x_0=h(\alpha_j)+O(HR/v_j)$. Since $M_0^2\gg t$, we also have
$v_jR\ll HR/v_j$, so that $x=h(\alpha_j)+O(HR/v_j)$.
Thus
$$
\frac{ty}{x}+2\pi\alpha_jy=\frac{2\pi y\alpha_j}{x}(x-h(\alpha_j))
\ll v_jR\cdot\frac{t}{M^2}\cdot\frac{HR}{v_j}\ll1.
$$
Finally, for $z$ near $x_0$, $\omega_j(z)$ is a linear combination of $1$,
$\sin^{s+1}(\frac{\pi}{2H}(z-\mathcal{N}_{j-1}))$,
and $\sin^{s+1}(\frac{\pi}{2H}(z-\mathcal{N}_j))$.
As above, we have $\frac{\pi y}{2H}\ll\frac{v_jR}{H}\ll1$,
so that
$$
\sin^{s+1}\!\left(\frac{\pi}{2H}(z-\mathcal{N})\right)\ll_s1
\quad\text{for }\mathcal{N}\in\{\mathcal{N}_{j-1},\mathcal{N}_j\},
$$
as desired.
\end{proof}
\subsection{Outline of the proof}  Since the proof of Proposition~\ref{NathanProp} is long, we give an outline.
The left-hand side of equation \eqref{keyidentity} 
is written as
$\mathcal{S} = 
   \sum_{j \in J(\beta,r)} c(f,r,\chi;j)  \mathcal{S}(j)$
where 
$$
    \mathcal{S}(j)= 2 \pi i^k   \sum_{\ell=1}^{\infty}  \lambda_{\bar{f}^{\chi}}(\ell) e( \tfrac{ \overline{ r a_j} \ell}{q_j})  I_j(\ell)
$$
and
\begin{equation}
   \label{Ijlintegral}
  I_j(\ell)
    = \frac{1}{q_j} \int_{0}^{\infty} F_j(x)  J_{k-1}(\tfrac{4 \pi
		\sqrt{\ell}}{q_j \sqrt{r}} \sqrt{x})\,\mathrm{d}x.
\end{equation}
Our goal is to develop an approximate formula for $\mathcal{S}$.  
This is done in five steps.
In the first four steps we determine an approximate formula for $\mathcal{S}(j)$, and in the final step these approximations are summed
over $j$ to obtain our formula for $\mathcal{S}$.  We choose real
parameters $K_1$ and $K$ such that $rK_1/d^2\le K$ and we decompose
$\cS(j) 
= \cS_{[1,K]}(j)+  \cS_{(K,\infty)}(j)$,
where for an interval $\mathcal{I} \subset \R$, 
$\cS_{\mathcal{I}}(j) :=  2 \pi i^k  \sum_{\ell \in I}  \lambda_{\bar{f}^{\chi}}(\ell) e( \tfrac{ \overline{ r a_j} \ell}{q_j})  I_j(\ell)$. 
Our bounds for the sums $\cS_{\mathcal{I}}(j)$ will depend on bounds 
for $\sum_{\ell \le x} | \lambda_{\bar{f}^{\chi}}(\ell)|$. These steps are now described in more precise detail.
\begin{enumerate}
\item[{\bf Step 1}.]   We first bound $ \cS_{(K,\infty)}(j)$.
For $\ell > K$,  the integral $I_j(\ell)$ is estimated by integration by parts, making use of   
the smoothness of $F_j$ and bounds for Bessel functions.
\item[{\bf Step 2}.]  Next we insert the asymptotic formula \cite[\textsection 8.451, Eqn.\,1]{gr}
\begin{equation} \label{BesselAsymp}
   J_{\nu}(x) = \sqrt{ \frac{2}{\pi x} } \cos\!\Big(x-\frac{\pi \nu}{2} -\frac{\pi}{4} \Big) + O\big( x^{-\frac{3}{2}}\big),
\end{equation}
which holds for $\nu \in \mathbb{Z}_{\ge 0}$ as $x \to \infty$, and estimate the corresponding error terms for $\ell \le K$ to deduce that
$\cS_{[1,K]}(j)$ equals $\tilde{\cS}_{[1,K]}(j)$ plus an error term, where $\tilde{\cS}_{[1,K]}(j)$ is a simplified sum.  We choose $K$ as a function of $v_j$, $M$, and $R$ to balance the error terms
in steps 1 and 2. 
\item[{\bf Step 3}.]  We are left with sums of the shape 
$\tilde{\cS}_{[1,K]}(j) =\sum_{\ell \le K}  \alpha_{\ell,j,r}
 \lambda_{\bar{f}^{\chi}}(\ell) e( \tfrac{ \overline{ r a_j} \ell}{q_j})  I^{\pm}_j(\ell)
$ 
where $\alpha_{\ell,j,r}\in\C$, $I^{\pm}_j(\ell)= \frac{1}{q_j}
\int_{0}^{\infty} F_j(x) e(\phi_{\pm}(x))\,\mathrm{d}x$, and $\phi_{\pm}(x)$ is a function
depending on parameters $\ell,j,t,r$.  We then choose $K_1$ so that 
$\phi'_{\pm}(x)$ does not change sign for $x \in \supp(F_j)$.  For those $\ell$ with $rK_1d^{-2} \le \ell \le K$, 
the integrals $I^{\pm}_j(\ell)$ are estimated using a weighted first derivative estimate (Lemma \ref{weightedfirstderiv}).
\item [{\bf Step 4}.]  Next, we treat the sum $\sum_{\ell \le rK_1 d^{-2}}  \alpha_{\ell,j,r}
 \lambda_{\bar{f}^{\chi}}(\ell) e( \tfrac{ \overline{ r a_j} \ell}{q_j})  I^{\pm}_j(\ell)$.  In this range of $\ell$ 
 the integrals $I^{\pm}_j(\ell)$ possess stationary points $x_{j}^{\pm}(\tfrac{\ell}{r})$.  Each integral is treated with 
Lemma \ref{stationaryphase}, 
leading to an expression
 $\mathcal{S}(j) = \mathcal{M}(j) +O(\tE_1(j) + \tE_2(j) + \tE_3(j))$ where 
 $\mathcal{M}(j)$ is a main term and the $\tE_i(j)$ are error terms.
\item [{\bf Step 5}.]  Finally, using $c(f,r,\chi;j) \ll_{N} 1$, we are left with the sum
 $$\mathcal{S} =  \sum_{j \in J(\beta,r)} c(f,r,\chi;j)  \mathcal{M}(j)  
 + O \Big(\sum_{j=1}^{J}\big(\tE_1(j) + \tE_2(j) + \tE_3(j)\big)\Big).$$
 In this last step, the error terms $\tE_i(j)$ are bounded as $j$ ranges over all Farey fractions. 
\end{enumerate}

\subsection{Proof of Proposition~\ref{NathanProp}} We now commence with the proof. 

\medskip

\noindent {\bf Step 1}. 
By repeated application of the identity \cite[\textsection 8.472, Eqn.\,3]{gr}
\[
\frac{\mathrm{d}}{ \mathrm{d}x} (x^{\nu} J_{\nu}(x)) = x^{\nu} J_{\nu-1}(x),
\]
we have
\begin{equation}
   \label{diffeq}
   \Big(  \frac{2}{A} \Big)^s  
    \frac{\mathrm{d}^{s}}{\mathrm{d}x^{s}}
    \Big( x^{\frac{k+s-1}{2}} J_{k+s-1}(A \sqrt{x}) \Big) =  x^{\frac{k-1}{2}} J_{k-1}(A \sqrt{x})
    \text{ for } s\ge0  \text{ and } A \ne 0. 
\end{equation}
Setting $A= \tfrac{4 \pi \sqrt{\ell}}{q_j \sqrt{r}}$,  we have
$I_j(\ell) = \frac{1}{q_j}  \int_{0}^{\infty} F_j(x)
		x^{-\frac{k-1}{2}}  (x^{\frac{k-1}{2}} J_{k-1}(A \sqrt{x}))\,\mathrm{d}x$.
Integrating by parts $s$ times and using \eqref{diffeq}, it follows that 
\[
   I_j(\ell) = \frac{(-1)^s}{q_j} \Big( \frac{2}{A} \Big)^s 
   \int_{0}^{\infty} 
    \frac{\mathrm{d}^{s}}{\mathrm{d}x^{s}}
    (  F_j(x) x^{-\frac{k-1}{2}}) 
   (x^{\frac{k+s-1}{2}} J_{k+s-1}(A \sqrt{x}))\,\mathrm{d}x.
\]
The asymptotic formula in \eqref{BesselAsymp} gives $J_{k+s-1}(A \sqrt{x})
\ll_{k,s} A^{-\frac{1}{2}} x^{-\frac{1}{4}}$. This estimate, along with Lemma~\ref{Andylemma}, implies that
\[
  |I_j(\ell)| \ll_{k,s}\frac{1}{q_j A^{s+\frac{1}{2}}(v_jR)^s}\int_{\supp(\omega_j)} 
  x^{\frac{2s-1}{4}}  \,\mathrm{d}x. 
\]
Using this bound for $\ell > K$ and  $A \asymp \ell^{\frac{1}{2}}
q_j^{-1} r^{-\frac{1}{2}}$, we deduce that
$$
\mathcal{S}_{(K,\infty)}(j)
 \ll_{k,s}  r^{\frac{s}{2}+\frac{1}{4}} q_j^{s-\frac{1}{2}} M^{\frac{2s-1}{4}}  \Big(\frac{1}{v_j R} \Big)^{s}
  \sum_{\ell > K} \frac{| \lambda_{\bar{f}^{\chi}}(\ell) |}{\ell^{\frac{s}{2}+\frac{1}{4}}} |\supp(\omega_j)|.
$$
Note that this last sum only converges for $s \ge 2$. 
By Lemma~\ref{l:L1average}(iii), we have
$$
  \mathcal{S}_{(K,\infty)}(j)  
  \ll_{k,N,s} r^{\frac{s}{2}+\frac{1}{4}} q_{j}^{s-\frac{1}{2}} M^{\frac{2s-1}{4}} \Big(\frac{1}{v_j R} \Big)^{s}
  |\supp(\omega_j)| K^{\frac{3}{4}-\frac{s}{2}} (\log K)^{-\delta}. 
$$
\noindent  {\bf Step 2}.  The asymptotic estimate in \eqref{BesselAsymp} 
implies that
$$
  J_{k-1} (\tfrac{4 \pi}{q} (\tfrac{\ell x}{r})^{\frac{1}{2}} )
  = \frac{r^{\frac{1}{4}}}{2 \sqrt{2} \pi} \frac{q_{j}^{\frac{1}{2}}}{ (\ell x)^{\frac{1}{4}}}
  \Big(  e(-\tfrac{k}{4})  e( 2 (\tfrac{\ell x}{r q_{j}^2})^{\frac{1}{2}} +\tfrac{1}{8})
  +  
  e(\tfrac{k}{4})   e(-2 (\tfrac{\ell x}{r q_{j}^2})^{\frac{1}{2}} -\tfrac{1}{8})
  \Big)
  + O_k\bigg( \frac{r^{\frac{3}{4} }q_{j}^{\frac{3}{2}
	}}{x^{\frac{3}{4}} \ell^{\frac{3}{4}}}  \bigg).
$$
Inserting this expression into \eqref{Ijlintegral} for each $\ell \le K$ and estimating the error terms, we have
$$
  \cS_{[1,K]}(j) =   \frac{i^kr^{\frac{1}{4}}}{\sqrt{2} q_{j}^{\frac{1}{2}} }
   \sum_{\ell \le K} \frac{ \lambda_{\bar{f}^{\chi}}(\ell) e( \tfrac{ \overline{ r a_j} \ell}{q_j})}{\ell^{\frac{1}{4}}}
  \sum_{\pm}  e (\pm \tfrac{k}{4}) \int_{0}^{\infty} F_j(x)  x^{-\frac{1}{4}}
   e(\mp 2(\tfrac{\ell x}{r q_{j}^2})^{\frac{1}{2}}  \mp \tfrac{1}{8})
	 \,\mathrm{d}x + \cE_0(j),
$$
where 
$$
  \cE_0(j) \ll \sum_{\ell \le K} \frac{ | \lambda_{\bar{f}^{\chi}}(\ell)|}{\ell^{\frac{3}{4}}} \frac{1}{q_j} 
  \int_{0}^{\infty} |F_j(x)| \frac{r^{\frac{3}{4} }q_{j}^{\frac{3}{2} }}{x^{\frac{3}{4}}}
  \,\mathrm{d}x
  \ll_{k,N} \frac{r^{\frac{3}{4} }q_{j}^{\frac{1}{2} }}{M^{\frac{3}{4}} }
 |\supp(\omega_j)|
    K^{\frac{1}{4}} (\log K)^{-\delta}
$$
by Lemma~\ref{l:L1average}(iv) since $x \asymp M$.  
Choosing $K:=(\frac{M}{v_jR})^{\frac2{s-1}}M_0$
and recalling that $|\supp(\omega_j)|\ll\frac{HR}{v_j}$, it follows that 
\begin{equation}
   \label{E1bound}
   \cE_1(j):=\mathcal{S}_{(K,\infty)}(j)+\cE_0(j) \ll_{k,N,s}\frac{HM^{\frac{2-s}{2(s-1)}} R^{\frac{s-2}{2(s-1)}}}{v_{j}^{\frac{s}{2(s-1)}}} 
  (\log M_0)^{-\delta}.
\end{equation}
Therefore 
$\mathcal{S}(j) = \tilde{\cS}_{[1,K]}(j) + \cE_1(j)$, 
where
\begin{align}
  \label{S0truncK}
  \tilde{\cS}_{[1,K]}(j)  & =  i^k  \sum_{\ell \le K} \frac{ \lambda_{\bar{f}^{\chi}}(\ell) e( \tfrac{ \overline{ r a_j} \ell}{q_j})}{\sqrt{ 2 (\frac{\ell}{r})^{\frac{1}{2}}q_j}}
   ( e(\tfrac{k}{4}) I_{j}^{+}(\ell) + e(-\tfrac{k}{4}) I_{j}^{-}(\ell) ), 
 \end{align}
   \begin{align}
   \label{Ijpml}  I_{j}^{\pm}(\ell)
   & = \int_{0}^{\infty} g_j(x) e(\phi_{\pm}(x)) \,\mathrm{d}x,
 \end{align}
\begin{align}
  \label{gdefn}
   g_j(x) & = x^{-\frac{1}{4}} \omega_j(x),
  \end{align}
and
\begin{align}
   \label{phipmdefn}
   \phi_{\pm}(x) & :=   -\frac{t}{2 \pi} \log x -\alpha_j x \mp 2
    (\tfrac{\ell x}{r q_{j}^2})^{\frac{1}{2}}  \mp \tfrac{1}{8}.
\end{align}

\noindent {\bf Step 3}.  
Let
$K_1=\big\lceil\frac{CM}{ R^2}\big\rceil$,
where $C$ is a sufficiently large positive constant.
For $\ell$ satisfying $rK_1 d^{-2} \le \ell \le K$, we bound the integral in \eqref{Ijpml} by showing that
$|\phi'_{\pm}(x)| \gg \frac{1}{q_j}  (  \frac{\ell}{rM} )^{\frac{1}{2}}$
and using a weighted first derivative estimate (Lemma \ref{weightedfirstderiv}).  It is convenient to write
$$\phi_{+}(x) ={\bf f}(x) +\frac{u_j}{v_j}x - \frac{2}{q_j} \Big(\frac{\ell
x}{r} \Big)^{\frac{1}{2}}-\frac{1}{8},$$
where ${\bf f}(x) = -\frac{t}{2 \pi} \log x$.
We shall make use of 
\begin{equation}
  \label{fderivatives}
   {\bf f}^{(j)}(x) = \frac{(j-1)!(-1)^{j}t}{2 \pi x^j} \asymp \frac{t}{M^j} \text{ for } x \in [M_1,M_2]
\end{equation}
for $j \ge 1$ and the identity
\begin{equation}
  \label{fhidentity}
 y= {\bf f}'(h(y)).
\end{equation} 
For $x \in \supp(\omega_j)$, the mean value theorem implies there exists $\xi\in\supp(\omega_j)$
such that
$${\bf f}'(x) = {\bf f}'(h(\rho_j)) + {\bf f}''(\xi)(x-h(\rho_j)) = \rho_j 
+ O \Big(\frac{t}{M^2} |\supp(\omega_j)| \Big),
$$
by \eqref{fderivatives} and \eqref{fhidentity}.
By  \eqref{fareydiff} and \eqref{omegajsupport} this is
$${\bf f}'(x) = \alpha_j + O \Big( \frac{1}{v_j R} +  \frac{t}{M^2} \frac{HR}{v_j} \Big).$$
Let $c_2$ be such that $M_2 \le c_2M$.  By the previous equation there exists $c_0 >0$ such that
\begin{equation}
\begin{split}
  \label{fprimealphabd}
  |{\bf f}'(x) - \alpha_j| & \le c_0 \Big(  \frac{1}{v_j R} +   \frac{t}{M^2} \frac{HR}{v_j} \Big)
  \le  \frac{2c_0}{q_jd  R}
   \le \frac{1 }{2q_j } \Big( \frac{\ell}{r  c_2M} \Big)^{\frac{1}{2}},
\end{split}
\end{equation} 
as long as $\frac{\ell}{r}\gg\frac{M}{d^2 R^2}$.  
Since $x \le M_2 \le c_2M$, we obtain  $ |{\bf f}'(x) - \alpha_j| \le \frac{1 }{2q_j } ( \frac{\ell}{r  x} )^{\frac{1}{2}}$.
It follows from \eqref{phipmdefn} and  \eqref{fprimealphabd} that for
$\frac{\ell}{r} \ge K_1 d^{-2}$,
$$
 |\phi_{+}'(x)| =|{\bf f}'(x)-\alpha_j -\tfrac{1}{q_j} (\tfrac{\ell}{rx} )^{\frac{1}{2}}| 
 \ge \tfrac{1}{2q_j} (\tfrac{\ell}{rx} )^{\frac{1}{2}}
 \gg  \tfrac{1}{2q_j} (\tfrac{\ell}{rM} )^{\frac{1}{2}}.
$$
We now compute the derivatives of
$F(x)=\phi_{+}(x)$ and
$G(x)=g_j(x)$ given by \eqref{gdefn}.  We have 
\begin{equation}
  \label{Fderivatives}
F^{(r_1)}(x) = {\bf f}^{(r_1)}(x) -  \frac{2}{q_j}  \sqrt{\frac{\ell}{r}}
 \frac{\mathrm{d}^{r_1}}{\mathrm{d}x^{r_1}} (x^{\frac{1}{2}})  \ll 
  \frac{t}{M^{r_1}}  + \frac{2}{q_j}  \sqrt{\frac{\ell}{r}} \frac{1}{M^{r_1-\frac{1}{2}}}
 \ll   \frac{t}{M^{r_1}}, 
\end{equation}
which follows (after some calculation) using the facts that $\ell \le K$ and $s \ge 6$.  Also we have 
\begin{equation}
  \label{Gderivatives}
  G^{(r_2)}(x)  = \sum_{i_1+i_2=r_2} \binom{r_2}{i_1} 
  \frac{\mathrm{d}^{i_1}}{\mathrm{d}x^{i_1}}
  x^{-\frac{1}{4}} \omega_j^{(i_2)}(x) 
  \ll \sum_{i_1+i_2=r_2}   M^{-\frac{1}{4}-i_1} H^{-i_2}
  \ll M^{-\frac{1}{4}} H^{-r_2}
\end{equation}
where we used $H=\frac{M^2}{R^2 t} \le M$.  We now invoke Lemma \ref{weightedfirstderiv} with $\alpha=\mathcal{N}_{j-1}-H$, $\beta=\mathcal{N}_j+H$, and $U=M^{-\frac{1}{4}}$ and make use of the lower bound
$F'(x) \gg \tfrac{1}{q_j} (\tfrac{\ell}{rM})^{\frac{1}{2}}$ for $x \in [\alpha,\beta]$.
With these choices the condition $M \ge \beta-\alpha$ is satisfied and $G(\alpha)=G(\beta)=0$. Thus,
for $rK_1 d^{-2} \le \ell \le K$, this lemma gives
\[
  I_{j}^{\pm}(\ell) \ll M^{-\frac{1}{4}} \frac{t}{M^2}
  \Big( 1 + \frac{M}{H} +\frac{M^2}{H^2} \frac{\tfrac{1}{q_j} (\tfrac{\ell}{rM})^{\frac{1}{2}} }{t/M} \Big)
  \Big(  \frac{q_j^2 r M}{\ell} \Big)^{\frac{3}{2}}.
\]
Since $H \le M$, we also have
\begin{align*}
  r^{\frac{1}{4}} q_{j}^{-\frac{1}{2}} \ell^{-\frac{1}{4}} I_{j}^{\pm}(\ell)
 & \ll r^{\frac{1}{4}} (q_{j}^2  \ell  M)^{-\frac{1}{4}}  \frac{t}{M^2} 
  \Big( \frac{M}{H} + \frac{M^3}{H^2 t}  \frac{1}{q_j} \Big(\frac{\ell}{rM}\Big)^{\frac{1}{2}}    
  \Big) \frac{q_{j}^3 r^{\frac{3}{2}} M^{\frac{3}{2}} }{ \ell^{\frac{3}{2}}} \\
     & \ll \frac{r^{\frac{7}{4}}}{\ell^{\frac{1}{2}}} 
     \Big(  
     q_{j}^{\frac{5}{2}} \Big( \frac{M}{\ell} \Big)^{\frac{5}{4}}  \frac{M}{H^2 R^2}
     +  q_{j}^{\frac{3}{2}} \Big( \frac{M}{\ell} \Big)^{\frac{3}{4}}  \frac{M}{H^2 \sqrt{r}}
     \Big),
\end{align*}
where we used the definition of $H$ in the last line. 
It follows that the contribution of the range $\ell \in (rK_1 d^{-2},K]$ to \eqref{S0truncK} is 
\begin{align*}
   \cE_2(j) &  \ll  r^{\frac{7}{4}} 
   \sum_{rK_1d^{-2} < \ell \le K} \frac{ |\lambda_{\bar{f}^{\chi}}(\ell)|}{\ell^{\frac{1}{2}}} 
    \Big(  
     q_{j}^{\frac{5}{2}} \Big( \frac{M}{\ell} \Big)^{\frac{5}{4}}  \frac{M}{H^2 R^2}
     +  q_{j}^{\frac{3}{2}} \Big( \frac{M}{\ell} \Big)^{\frac{3}{4}}  \frac{M}{H^2 \sqrt{r}}
     \Big) \\
     & \ll  r^{\frac{7}{4}}  
     \Big( 
      \frac{q_{j}^{\frac{5}{2}} M^{\frac{9}{4}}}{H^2 R^2}   \sum_{\ell>rK_1 d^{-2}} \frac{  |\lambda_{\bar{f}^{\chi}}(\ell)| }{  \ell^{\frac{7}{4}}}
     +   \frac{q_{j}^{\frac{3}{2}} M^{\frac{7}{4}}}{H^2 \sqrt{r}}  \sum_{\ell>rK_1 d^{-2}} \frac{  |\lambda_{\bar{f}^{\chi}}(\ell)| }{  \ell^{\frac{5}{4}}} 
     \Big).
\end{align*}
Using Lemma~\ref{l:L1average}(iii) and the estimate $K_1 \asymp M/R^2$, we have
\begin{equation}\label{E2bound}
\begin{aligned}
  \cE_2(j) & \ll_{k,N} r^{\frac{7}{4}}
   \Big( 
     \frac{q_{j}^{\frac{5}{2}} M^{\frac{9}{4}}}{H^2 R^2}  
     \Big( \frac{rM}{R^2 d^2} \Big)^{-\frac{3}{4}}
     + \frac{q_{j}^{\frac{3}{2}} M^{\frac{7}{4}}}{H^2 \sqrt{r}}
      \Big( \frac{rM}{R^2 d^2} \Big)^{-\frac{1}{4}}    
     \Big)(\log (2+rK_1 d^{-2}))^{-\delta} \\
     & \ll_{k,N} \frac{v_{j}^{\frac{3}{2}} M^{\frac{3}{2}}}{H^2} 
     \Big(\frac{v_j}{R^{\frac{1}{2}}} +  R^{\frac{1}{2}}  \Big)
		 (\log M_0)^{-\delta} \ll
	\Big( \frac{v_j}{R} \Big)^{\frac{3}{2}} \frac{M^{\frac{3}{2}} R^2}{H^2} 
    (\log M_0)^{-\delta},
\end{aligned}
\end{equation}
since $v_j\le R$.\\

\noindent {\bf Step 4}.   We have shown that
$\mathcal{S}(j) = \tilde{\cS}_{[1,rK_1d^{-2}]}(j) + \cE_1(j) +\cE_2(j)$
where  
\begin{align}
  \label{S0truncK2}
 \tilde{\cS}_{[1,rK_1d^{-2}]}(j)  & =  i^k  \sum_{\ell \le rK_1d^{-2}} \frac{ \lambda_{\bar{f}^{\chi}}(\ell) e( \tfrac{ \overline{ r a_j} \ell}{q_j})}{\sqrt{ 2 (\frac{\ell}{r})^{\frac{1}{2}}q_j}}
   ( e(\tfrac{k}{4}) I_{j}^{+}(\ell) + e(-\tfrac{k}{4}) I_{j}^{-}(\ell) ), 
\end{align}
and $\cE_1(j)$ and $\cE_2(j)$ are estimated by \eqref{E1bound} and \eqref{E2bound}, respectively.
For $\ell \le rK_1d^{-2}$, we extract the stationary phase terms of the integrals $I_{j}^{\pm}(\ell)$
given by \eqref{Ijpml}. 
Let $x_{j}^{\pm}(\ell)$ be the roots of 
$$
     \frac{\mathrm{d}}{\mathrm{d}x}
   \Big(-\frac{t}{2 \pi} \log x -\alpha_j x \mp 2
    (\tfrac{lx}{q_{j}^2})^{\frac{1}{2}}  \mp \tfrac{1}{8} \Big) = 0. 
$$
Notice that the numbers $x_{j}^{\pm}(\frac{\ell}{r})$ are the stationary 
points satisfying 
$\phi_{\pm}'(x_{j}^{\pm}(\tfrac{\ell}{r})) = 0$.
It follows that the 
$x_{j}^{\pm}(\ell)$ are the positive roots of 
\begin{equation}
  \label{stationary}
  -\frac{t}{2 \pi x}+\frac{u_j}{v_j} \mp 
    \Big(\frac{\ell}{ x} \Big)^{\frac{1}{2}}  \frac{1}{q_j}  =0,
\end{equation} 
since $\alpha_j=-\frac{u_j}{v_j}$.  Multiplying by $x$ this becomes
$\frac{u_j}{v_j} x \mp \frac{\ell^{\frac{1}{2}}}{q_j} \sqrt{x} -\frac{t}{2 \pi} = 0$ so that
$$\sqrt{x} = \frac{ \pm \frac{\ell^{\frac{1}{2}}}{q_j} \pm \sqrt{ \frac{\ell}{q_j^2}+ 4 \frac{u_j}{v_j} \frac{t}{2 \pi}}  }{2 \frac{u_j}{v_j} }.
$$
We discard those solutions corresponding to the second $-$ sign since $\sqrt{x}$ is necessarily positive.
With a little calculation, it follows that 
\begin{align*}
  x_{j}^{\pm}(\ell) = \left(  \frac{ \pm \frac{\ell^{\frac{1}{2}}}{q_j} + \sqrt{ \frac{\ell}{q_j^2}+  \frac{2u_j q_{j}^2 t/ \pi v_j}{q_j^2}} }{2 \frac{u_j}{v_j} }  \right)^2  
  &=   \Big( \frac{d}{2u_j} \Big)^2 \Bigg( \sqrt{\ell + \frac{2 u_j  q_j t}{ \pi d} } 
  \pm \sqrt{\ell} \Bigg)^2,
\end{align*}
since $v_j=d q_j$.  
Finally, we apply the stationary phase lemma (Lemma \ref{stationaryphase}) 
to those $I_{j}^{+}(\ell)$ with $\ell \le rK_1d^{-2}$.  We choose 
$F(x) =\phi_{+}(x)$, $G(x)=g_j(x)$,  $\alpha=x_{j}^{+}(\tfrac{\ell}{r})-\frac{M}{4}$, $\beta=x_{j}^{+}(\tfrac{\ell}{r})+\frac{M}{4}$,
and $\gamma =x_{j}^{+}(\tfrac{\ell}{r})$.
We also have the parameters $t,M,H$, and $U=M^{-\frac{1}{4}}$ which correspond to those of Lemma \ref{stationaryphase} and 
we have the derivative bounds \eqref{Fderivatives} and \eqref{Gderivatives}. 
Observe that the conditions $M \ge \beta-\alpha$ and $H \ge M/\sqrt{t}$ are both met.
With these choices we now demonstrate that $\supp(\omega_j)=[\mathcal{N}_{j-1}-H,\mathcal{N}_j+H] \subseteq [\alpha,\beta]$
so that $G(\alpha)=G(\beta)=0$.  We aim to show $\alpha=x_{j}^{+}(\tfrac{\ell}{r})-\frac{M}{4}
\le \mathcal{N}_{j-1}-H$.  By the mean value theorem there exists $\xi\in\supp(\omega_j)$ 
such that
\begin{equation}
  \label{fprimelb}
  |{\bf f}'(\mathcal{N}_{j-1})-{\bf f}'(x_{j}^{+}(\tfrac{\ell}{r}))| = |{\bf f}''(\xi)|| \mathcal{N}_{j-1}-x_{j}^{+}(\tfrac{\ell}{r})|
  \gg \frac{t}{M^2} | \mathcal{N}_{j-1}-x_{j}^{+}(\tfrac{\ell}{r})|.
\end{equation}
Similarly there exists $\xi'\in\supp(\omega_j)$ such that
$$
{\bf f}'(\mathcal{N}_{j-1}) ={\bf f}'(h(\rho_{j-1})) + O({\bf f}''(\xi' )) =\rho_{j-1} + O \Big( \frac{t}{M^2} \Big),
$$ 
by \eqref{fhidentity}. By \eqref{stationary}
$
  {\bf f}'(x_{j}^{+}(\tfrac{\ell}{r})) = -\frac{t}{2 \pi x_{j}^{+}(\tfrac{\ell}{r})  } = \alpha_j \pm 
    \Big(\frac{\ell/r}{ x_{j}^{+}(\tfrac{\ell}{r})} \Big)^{\frac{1}{2}}  \frac{1}{q_j}$.
Using \eqref{fareydiff} it follows that
\begin{equation}
  \label{fprimeub}
  |{\bf f}'(\mathcal{N}_{j-1}) -{\bf f}'(x_{j}^{+}(\tfrac{\ell}{r}))| \ll \frac{1}{v_j R} + \frac{t}{M^2} +  \Big(\frac{\ell/r}{ M} \Big)^{\frac{1}{2}}  \frac{1}{q_j}
  \ll  \frac{1}{v_jR}   + \frac{t}{M^2}, 
\end{equation}
since $\frac{\ell}{r} \le \frac{CM}{d^2 R^2}$.  Combining \eqref{fprimelb} and \eqref{fprimeub} yields
$| \mathcal{N}_{j-1}-x_{j}^{+}(\tfrac{\ell}{r})| \ll \frac{M^2}{t R} + 1$
and thus 
$$
 x_{j}^{+}(\tfrac{\ell}{r})- \mathcal{N}_{j-1} +H \le O\Big( \frac{M^2}{t R} + 1  \Big)+\frac{M^2}{R^2 t} 
 \le \frac{M}{4},
$$
assuming that $t_0$ is sufficiently large.
Hence $\alpha\le\mathcal{N}_{j-1}-H$, and an analogous
argument establishes that $\mathcal{N}_{j}+H\le\beta$. 
The stationary point of $G=\phi_{+}$ is $x_{j}^{+}(\tfrac{\ell}{r})$.
Hence the main term in Lemma \ref{stationaryphase} is
$$
   \frac{x_{j}^{+}(\tfrac{\ell}{r})^{-\frac{1}{4}}
	 \omega_j(x_{j}^{+}(\tfrac{\ell}{r}))e(\phi_{+}(x_{j}^{+}(\tfrac{\ell}{r}))+\frac{1}{8})}{\sqrt{
	 \phi_{+}^{''}(x_{j}^{+}(\tfrac{\ell}{r}) ) }},
$$
and since $\gamma-\alpha=\beta-\gamma=\frac{M}{4}$, the error term is 
$$
  \ll 
   \frac{M^4 M^{-\frac{1}{4}}}{t^{2}}\Big( 1+ \frac{M}{H} \Big)^2
  M^{-3}  
    +  \frac{M M^{-\frac{1}{4}}}{t^{\frac{3}{2}}} \Big( 1+ \frac{M}{H} \Big)^2
    \ll  M^{-\frac{1}{4}}  \frac{M^3}{t^{\frac{3}{2}} H^2},
$$
as the second error term dominates the first and  $H \le M$.  A similar argument establishes the analogous result for $I_{j}^{-}(\ell)$.  
Thus,
$$
   I_{j}^{\pm}(\ell) =  \frac{x_{j}^{\pm}(\tfrac{\ell}{r})^{-\frac{1}{4}} \omega_j(x_{j}^{\pm}(\tfrac{\ell}{r}))e(\phi_{\pm}(x_{j}^{\pm}(\tfrac{\ell}{r}))+\frac{1}{8})}{\sqrt{ \phi_{\pm}^{''}(x_{j}^{\pm}(\tfrac{\ell}{r}) ) }}  
  + O\Big( \frac{M^{\frac{11}{4}}}{t^{\frac{3}{2}} H^2} \Big)
  \text{ for } \ell \le rK_1d^{-2}.
$$
The error term, when inserted into \eqref{S0truncK2}, becomes
\begin{align*}
  \cE_3(j) \ll \frac{r^{\frac{1}{4}}}{q_{j}^{\frac{1}{2}}}
  \sum_{\ell\le rK_1 d^{-2}}  
  \frac{|\lambda_{\bar{f}^{\chi}}(\ell)|  M^{\frac{11}{4}}}{\ell^{\frac{1}{4}} H^2 t^{\frac{3}{2}}}
   & \ll_{k,N}\frac{r^{\frac{1}{4}}M^{\frac{11}{4}}}{
	 q_{j}^{\frac{1}{2}} H^2 t^{\frac{3}{2}} }  \Big(\frac{rM}{R^2 d^2}
	 \Big)^{\frac{3}{4}}(\log M_0)^{-\delta},
\end{align*}
by Lemma~\ref{l:L1average}(iv) and using $K_1 \asymp \frac{M}{R^2}$.  
It follows that
$$
    \cE_3(j) \ll_{k,N} \frac{M^{\frac{7}{2}}(\log M_0)^{-\delta}}{q_{j}^{\frac{1}{2}} H^2
		t^{\frac{3}{2}} R^{\frac{3}{2}}}  
    \ll_{k,N} \frac{M^{\frac{1}{2}} R^{\frac{3}{2}}}{\sqrt{Hv_j}}
		(\log M_0)^{-\delta},
$$
since $H \asymp \frac{M^2}{R^2 t}$ and $v_j=dq_j$. 
Hence we have established
$$
\mathcal{S}(j)=\mathcal{M}(j) +
O_{k,N,s}\Bigl(\bigl(\tE_1(j)+\tE_2(j)+\tE_3(j)\bigr)
(\log M_0)^{-\delta}\Bigr),
$$
where
\begin{equation}\label{Mj}
\begin{aligned}
   \mathcal{M}(j) =  
   & i^ke(\tfrac{k}{4})
   \sum_{\ell \le rK_1 d^{-2}} 
   \frac{ \lambda_{\bar{f}^{\chi}}(l) e( \tfrac{ \overline{ r a_j} l}{q_j})\omega_j(x_{j}^{+}(\tfrac{l}{r}))e(\phi_{+}(x_{j}^{+}(\tfrac{l}{r}))+\frac{1}{8})}{\sqrt{ 2 (\frac{\ell}{r})^{\frac{1}{2}}q_j} x_{j}^{+}(\tfrac{l}{r})^{\frac{1}{4}}\sqrt{ \phi_{+}^{''}(x_{j}^{+}(\tfrac{l}{r}) ) }}
    \\
    &+i^ke(-\tfrac{k}{4})  \sum_{\ell \le rK_1 d^{-2}} \frac{
		\lambda_{\bar{f}^{\chi}}(l) e( \tfrac{ \overline{ r a_j}
		l}{q_j})\omega_j(x_{j}^{-}(\tfrac{l}{r}))e(\phi_{-}(x_{j}^{-}(\tfrac{l}{r}))+\frac{1}{8})}{\sqrt{
		2 (\frac{\ell}{r})^{\frac{1}{2}}q_j}
		x_{j}^{-}(\tfrac{l}{r})^{\frac{1}{4}}\sqrt{
		\phi_{-}^{''}(x_{j}^{-}(\tfrac{l}{r}) ) }},
\end{aligned}
\end{equation}
$
     \tE_1(j) =
     \frac{HM^{\frac{2-s}{2(s-1)}}
		 R^{\frac{s-2}{2(s-1)}}}{v_{j}^{\frac{s}{2(s-1)}}}$,  
   $\tE_2(j) = ( \frac{v_j}{R} )^{\frac{3}{2}}
	 \frac{M^{\frac{3}{2}} R^2}{H^2}$, and 
    $\tE_3(j) = \frac{M^{\frac{1}{2}} R^{\frac{3}{2}}}{\sqrt{Hv_j}}$.
We now simplify the expression for $\mathcal{M}(j)$. 
By \eqref{phipmdefn}, it follows that 
$$
x^{\frac{1}{4}}  (\phi_{\pm}^{''}(x ))^{\frac{1}{2}} =
   (x^{\frac{1}{2}})^{\frac{1}{2}}
   \Big(  \frac{t}{2 \pi x^2} \pm \frac{1}{2 q_j} \sqrt{\frac{\ell}{r}} x^{-\frac{3}{2}} \Big)^{\frac{1}{2}}
   =   \Big(  \frac{t}{2 \pi x^{\frac{3}{2}}} \pm \frac{1}{2 q_j}
	 \sqrt{\frac{\ell}{r}} x^{-1} \Big)^{\frac{1}{2}},
$$
and thus
$\sqrt{ 2 (\frac{\ell}{r})^{\frac{1}{2}}q_j  }x_{j}^{\pm}(\tfrac{\ell}{r})^{\frac{1}{4}}  (\phi_{\pm}^{''}(x_j^{\pm}(\tfrac{\ell}{r}) ))^{\frac{1}{2}} =
  \Big( \frac{ \sqrt{\frac{\ell}{r}} q_j t  }{\pi x_j^{\pm}(\tfrac{\ell}{r})^{\frac{3}{2}}}  \pm \frac{\ell}{r} \frac{1}{x_j^{\pm}(\tfrac{\ell}{r})} \Big)^{\frac{1}{2}}
$.
Since $i^ke(\pm \tfrac{k}{4})=(\mp 1)^k$, the expression in \eqref{Mj} simplifies to
\begin{equation}
 \label{M0j}
  \mathcal{M}(j) =    \sum_{\pm} (\mp 1)^k \sum_{\ell=1}^{rK_1 d^{-2}}  \lambda_{\bar{f}^{\chi}}(\ell) e( \tfrac{ \overline{ r a_j} \ell}{q_j})  
    \omega_j(x_{j}^{\pm}(\tfrac{\ell}{r})) h_{j}^{\pm}(\tfrac{\ell}{r})
		e(g_{j}^{\pm}(\tfrac{\ell}{r})),
\end{equation}
where $x_{j}^{\pm}(\ell)$, $h_{j}^{\pm}(\ell)$, and $g_{j}^{\pm}(\ell)$
are given by \eqref{xjlpm}, \eqref{gjlpm}, and \eqref{hjlpm}, respectively.
Therefore,
\begin{equation}
\begin{split}
  \label{SEjsums}
  \mathcal{S}  & =  
   \sum_{j \in J(\beta,r)} c(f,r,\chi;j) \mathcal{M}(j) 
   + O_{k,N,s} \Big(  (\log M_0)^{-\delta} \sum_{j=1}^{J}
	 \big(\tE_1(j)+\tE_2(j)+\tE_3(j)\big)\Big), 
\end{split}
\end{equation}
since $\sum_{j \in J(\beta,r)}\tE_i(j)\le \sum_{j=1}^J\tE_i(j)$.  \\

\noindent {\bf Step 5}.  
In this final step, we bound the error terms in \eqref{SEjsums}.  
First, we divide the sum over $j$ into subsums where the $v_j$ lie 
in dyadic intervals $[Q,2Q]$ where $Q=2^i$, $i \ge 0$, and $Q \le R$.  
We require a bound for the number of $v_j$ in $[Q,2Q]$.  
Observe that the Farey fractions $-\frac{u_j}{v_j}$ lie in the interval $
 \mathcal{I} = [-\tfrac{t}{2 \pi(M_1+2H)},
 -\tfrac{t}{2\pi(M_2-2H)}]$ of length $|\mathcal{I}| \asymp \frac{t}{M}
 \asymp \frac{M}{HR^2}$, since $H \asymp \frac{M^2}{R^2 t}$.  
Now if $\mathcal{F}(Q)$ denotes the extended Farey fractions with denominator less than or equal to $Q$, then \cite[Lemma 1.2.3]{Huxley}
gives
\[
   \sum_{\alpha \in \mathcal{F}(Q) \cap I} 1 \le \Delta Q^2 +1,
\]
where $I$ is an interval of length $\Delta$.  Applying this estimate with $\Delta = \frac{M}{HR^2}$, we have 
\begin{equation}
    \label{vjdyadic}
    \sum_{Q \le v_j \le 2Q}  1  \le \sum_{ \alpha \in \mathcal{F}(2Q) \cap \mathcal{I} }  1
     \ll 
\frac{MQ^2}{HR^2}+1.
\end{equation}
Consequently,
\begin{align*}
  \sum_{j=1}^{J}\tE_1(j) 
  &  \ll  H\sum_{Q \le R}  
   M^{\frac{2-s}{2(s-1)}} R^{\frac{s-2}{2(s-1)}}   \Big( \sum_{Q \le v_j \le 2Q} 
 v_{j}^{-\frac{s}{2(s-1)} } \Big)   \\
 &  \ll  H\sum_{Q \le R}   \frac{M^{\frac{2-s}{2(s-1)}} R^{\frac{s-2}{2(s-1)}}}{ Q^{\frac{s}{2(s-1)}}} \Big( \frac{MQ^2}{HR^2}+1 \Big)   \\
 & =  H\sqrt{\frac{R}{M}}   \Big( \frac{M}{R} \Big)^{\frac{1}{2(s-1)}}
   \Big(  
   \frac{M}{HR^2} \sum_{Q \le R}   Q^{\frac{3}{2}-\frac{1}{2(s-1)}}
   + \sum_{Q \le R} Q^{-\frac{1}{2}-\frac{1}{2(s-1)}}
   \Big),
\end{align*}
Using the elementary estimate
\begin{equation}
     \label{Qsums}
       \sum_{\substack{Q \le R \\ Q =2^i, i \ge 0}} Q^{c_1} \ll_{c_1} 
       \begin{cases}
       R^{c_1}, & \text{ for } c_1 > 0, \\
       1, & \text{ for } c_1 < 0,
       \end{cases}
\end{equation}
with $c_1=\frac{3}{2}-\frac{s}{2(s-1)}$ and $c_1 = -\frac{1}{2}-\frac{s}{2(s-1)}$ (for $s\ge2$)  it follows that
\begin{align*}
   \sum_{j=1}^{J} \tE_1(j) &  \ll
     H\sqrt{\frac{R}{M}}   \Big( \frac{M}{R} \Big)^{\frac{1}{2(s-1)}}
   \Big(  
   \frac{M}{HR^2}    R^{\frac{3}{2}-\frac{1}{2(s-1)}}
   +1
   \Big)  \\
    & \ll \sqrt{\frac{R}{M}}   \Big( \frac{M}{R} \Big)^{\frac{1}{2(s-1)}}
   \Big(  
   \frac{M}{R^2}    R^{\frac{3}{2}-\frac{1}{2(s-1)}}
   +H
   \Big).
\end{align*}
Observe that $H\ll\frac{M^2}{R^2t}\ll\frac{M}{R}$. Since $s\ge6$, the second term in the brackets is bounded by the
first, and 
\begin{equation}
 \label{sumE1bd}
     \sum_{j=1}^{J} \tE_1(j)   \ll\sqrt{\frac{R}{M}}   \Big( \frac{M}{R} \Big)^{\frac{1}{2(s-1)}} 
   \frac{M}{R^2}    R^{\frac{3}{2}-\frac{1}{2(s-1)}}  =  \sqrt{M}  \Big( \frac{M}{R^2} \Big)^{\frac{1}{2(s-1)}}.
\end{equation}
Turning to the second error term in \eqref{SEjsums}, we have
\begin{align*}
  \sum_{j=1}^{J} \tE_2(j) 
      & \ll   \frac{M^{\frac{3}{2}} R^2}{H^2} 
      \sum_{Q \le R}
     \sum_{Q \le v_j \le 2Q}  \Big( \frac{v_j}{R} \Big)^{\frac{3}{2}} 
      \ll     \frac{M^{\frac{3}{2}} R^2}{H^2} 
     \sum_{Q \le R}
   \Big( \frac{Q}{R}  \Big)^{\frac{3}{2}}  
    \Big(
    \frac{MQ^2}{HR^2} +1
    \Big).
\end{align*}
Again applying \eqref{Qsums} with $c_1=\frac{7}{2}$ 
and $c_1=\frac{3}{2}$, we find that 
\begin{equation}
\begin{split}
   \label{sumE2bd}
    \sum_{j=1}^{J} \tE_2(j)   & \ll\frac{M^{\frac{3}{2}} R^2}{H^2} 
    \Big(
    \frac{M}{H}   
     \sum_{Q \le R}
   \Big( \frac{Q}{R}  \Big)^{\frac{7}{2}}  
    +   \sum_{Q \le R}
  \Big( \frac{Q}{R}  \Big)^{\frac{3}{2}}  
    \Big) 
    \ll     \frac{M^{\frac{3}{2}} R^2}{H^2} 
    \Big(
    \frac{M}{H}      
    + 1
    \Big)   \ll   \frac{M^{\frac{5}{2}} R^2}{H^3},
\end{split}
\end{equation}
since $H\ll\frac{M}{R}\ll M$.
The third error term in \eqref{SEjsums} is 
\begin{align*}
   \sum_{j=1}^{J} \tE_3(j) 
   \ll 
   \frac{M^{\frac{1}{2}} R^{\frac{3}{2}}}{\sqrt{H}}  
    \sum_{Q \le R}
     \sum_{Q \le v_j \le 2Q} v_{j}^{-\frac{1}{2}} 
  &  \ll 
   \frac{M^{\frac{1}{2}} R^{\frac{3}{2}}}{\sqrt{H}}  
    \sum_{Q \le R} Q^{-\frac{1}{2}}
       \Big(
    \frac{MQ^2}{HR^2} +1
    \Big),
\end{align*}
by \eqref{vjdyadic}. By \eqref{Qsums} with $c_1=\frac{3}{2}$
and $c_2=-\frac{1}{2}$, we have
\begin{equation}
\begin{split}
  \label{sumE3bd}
   \sum_{j=1}^{J} \tE_3(j)
   \ll 
   \frac{M^{\frac{1}{2}} R^{\frac{3}{2}}}{\sqrt{H}}  
   \Big(   \frac{M}{HR^2}
  R^{\frac{3}{2}}
   +  1
     \Big) 
      &= \frac{M^{\frac{3}{2}} R}{H^{\frac{3}{2}}}  \Big( 1
     + \frac{R^{\frac{1}{2}} H}{M}
      \Big)       \ll  \frac{M^{\frac{3}{2}} R}{H^{\frac{3}{2}}}  ,
\end{split}
\end{equation} 
since $\frac{R^{\frac{1}{2}} H}{M}
\ll R^{-\frac{1}{2}} \ll 1$.
Collecting the estimates in \eqref{sumE1bd}, \eqref{sumE2bd}, and
\eqref{sumE3bd}, we find that
\begin{equation}\label{sum of errors}
  \sum_{j=1}^{J}\big(\tE_1(j)+\tE_2(j) +\tE_3(j)\big) 
 \ll \sqrt{M} \Big( \frac{M}{R^2} \Big)^{\frac{1}{2(s-1)}}
 + \frac{M^{\frac{5}{2}} R^2}{H^3}  
 +\frac{M^{\frac{3}{2}} R}{H^{\frac{3}{2}}}.
\end{equation}
Note that the third error term is dominated by the first two. 
To see this, note that if $MR \ge H^{\frac{3}{2}}$, then
$\frac{M^{\frac{3}{2}} R}{H^{\frac{3}{2}}} \le
    \frac{M^{\frac{5}{2}} R^2}{H^3}$ while if $MR \le H^{\frac{3}{2}}$, then
$\frac{M^{\frac{3}{2}} R}{H^{\frac{3}{2}}}
   =\sqrt{M} \frac{MR}{H^{\frac{3}{2}}} \le \sqrt{M} \ll 
   \sqrt{M} ( \frac{M}{R^2} )^{\frac{1}{2(s-1)}}$.
Therefore the right-hand side of \eqref{sum of errors} is 
$ O ( ( \sqrt{M} ( \frac{M}{R^2} )^{\frac{1}{2(s-1)}}
 + \frac{M^{\frac{5}{2}} R^2}{H^3}) (\log M_0)^{-\delta})$.
Proposition \ref{NathanProp} now follows by combining \eqref{M0j}, \eqref{SEjsums}, and \eqref{sum of errors}.

\section{Proof of Proposition~\ref{Large Sieve}}\label{s:largesieve}
Define the functions
\[
g_{jr}^\pm(\ell) = g_j^\pm(\ell/r) + \frac{\overline{ra_j} \ell}{ q_j}, \quad h_{jr}^\pm(\ell) = \frac{\ell}{r} h_j^{\pm}(\ell/r),
 \]
\[
H_{ijr}^\pm(\ell) =  h_{ir}^{\pm}(\ell) h_{jr}^{\pm}(\ell), \quad \text{and} \quad W_{ijr}^\pm(\ell) =  \omega_i\big( x_i^\pm(\ell/r) \big) \omega_j\big( x_j^\pm(\ell/r) \big),
\]
where we recall that $g_{j}^{\pm}, h_j^{\pm}, x_{j}^{\pm}$, and $\omega_j$ are given 
by \eqref{gjlpm}, \eqref{hjlpm}, \eqref{xjlpm}, and \eqref{omegaj}, respectively. 
Note that $|W_{ijr}^\pm(\ell)|\le 1$ and, since $ x_j^\pm(\ell)$
is monotonic, $W_{ijr}^\pm(\ell)$ has bounded variation (over all of
$\R$). Applying Cauchy's inequality in the $\ell$ variable in
\eqref{large_sieve_estimate} and then expanding out the resulting square,
we have
\[
\begin{split}
\Bigg| \sum_{\ell=L_1}^{L_2}&\frac{r\lambda(\ell)}{\ell}
\sum_{ \substack{j\in J(\beta,r) \\ (u_j,v_j)\in\mathcal{R}} }
\nu(j) e\big( g_{jr}^\pm(\ell) \big) h_{jr}^\pm(\ell)   \omega_j\big( x_j^\pm(\ell/r) \big) \Bigg|^2
\\
&\le  r^2 \sum_{\ell=L_1}^{L_2}  \frac{ |\lambda(\ell)|^2}{\ell^2}
\mathop{\sum \sum}_{\substack{i,j \in J(\beta,r) \\(u_i,v_i),(u_j,v_j)\in\mathcal{R}}  } \nu(i) \overline{\nu(j)} \sum_{\ell=L_1}^{L_2} W_{ijr}^\pm(\ell) H_{ijr}^\pm(\ell) e\big( g_{ir}^\pm(\ell) - g_{jr}^\pm(\ell) \big)
\\
&\ll\frac{r^2}{L^2}
\sum_{\ell=L_1}^{L_2}|\lambda(\ell)|^2
\max_{j\in J(\beta,r)}|\nu(j)|^2\cdot
\mathop{\sum \sum}_{\substack{i,j \in J(\beta,r) \\(u_i,v_i),(u_j,v_j)\in\mathcal{R}}  } \left|\sum_{\ell=L_1}^{L_2} W_{ijr}^\pm(\ell) H_{ijr}^\pm(\ell) e\big( g_{ir}^\pm(\ell) - g_{jr}^\pm(\ell) \big) \right|.
\end{split}
\]
Let
\[
Y_{ijr}^\pm(L_1,L_2) = \big|W_{ijr}^\pm(L_1) H_{ijr}^\pm(L_1)\big| + \int_{L_1}^{L_2}\left| \frac{\mathrm{d}}{\mathrm{d}x} W_{ijr}^\pm(x) H_{ijr}^\pm(x)  \right| \mathrm{d}x.
\]
We now apply \cite[Lemma~5.1.1]{Huxley} to see that 
\[
\begin{split}
\Bigg|\sum_{\ell=L_1}^{L_2} W_{ijr}^\pm(\ell) H_{ijr}^\pm(\ell) & e\big( g_{ir}^\pm(\ell) - g_{jr}^\pm(\ell) \big) \Bigg| \ \le \ Y_{ijr}^\pm(L_1,L_2)  \max_{L_1^\prime \in [L_1,L_2]}
\left|\sum_{\ell=L_1^\prime}^{L_2} e\big( g_{ir}^\pm(\ell) - g_{jr}^\pm(\ell) \big) \right|.
\end{split}
\]

We first estimate  $Y_{ijr}^\pm(L_1,L_2)$. For any function $W$ of bounded variation, we have
\[
\int \Big|\frac{\mathrm{d}}{\mathrm{d}x} (WH) \Big| \, \mathrm{d}x \, \ll \, \max|H| + \int |H^\prime(x)| \, \mathrm{d}x. 
\]
Therefore
\[
Y_{ijr}^\pm(L_1,L_2) \ll \max_{\ell\in[L_1,L_2]} H_{ijr}^\pm(\ell) +  \int_{L_1}^{L_2} \Big|\frac{\mathrm{d}}{\mathrm{d}x} H_{ijr}^\pm(x) \Big|  \, \mathrm{d}x,
\]
since $H_{ijr}^\pm(\ell)$ is positive.
We now estimate the right-hand side of this inequality. Observe that
\[
h_{jr}^\pm(\ell) = q_j \sqrt{\frac{\ell}{r q_j^2} x_j^\pm(\ell/r) } \cdot
\left[ \frac{t}{\pi \sqrt{\frac{\ell}{r q_j^2} x_j^\pm(\ell/r) }} \pm 1
\right]^{-\frac12}.
\]
Let $y_j=\frac{d}{2r u_jq_j}$. Then 
\[
\sqrt{\frac{\ell}{r q_j^2} x_j^\pm(\ell/r) } =  \sqrt{ (\ell y_j)^2 + \frac{t \ell y_j}{\pi} } \pm  \ell y_j = \frac{\frac{t y_j}{\pi}}{\sqrt{y_j^2 + \frac{ty_j}{\pi\ell} } \mp  y_j}.
\]
From the third formula, we see that this is an increasing function of $\ell$ for $\ell > 0$ and thus $h_{jr}^\pm(\ell)$ is an increasing function of $\ell$ for $\ell > 0$, as well. Therefore
$Y_{ijr}^\pm(L_1,L_2) \ll H_{ijr}^\pm(L_2)$.
Since $R=\sqrt{M/M_0}$ and $M\ll \sqrt{C}$, by \eqref{tbound} and \eqref{Mbound} we have $M \ll tR$. This implies that 
\[
d^2 L \ll \frac{rV^2M}{R^2} \asymp \frac{r UV M^2}{tR^2} \ll rtUV
\]
so that $\frac{t}{\pi \ell y_j} \gg 1$ for all $j$. Therefore
\[
\sqrt{\frac{\ell}{r q_j^2} x_j^\pm(\ell/r) } = \sqrt{ (\ell y_j)^2 + \frac{t \ell y_j}{\pi} } \pm \ell y_j \asymp
\eta{L}
\]
and thus
\[
h_{jr}^\pm(\ell) \asymp
\frac{V}{d}\sqrt{\frac{(\eta L)^3}{t}}.
\]
This means that 
\[
Y_{ijr}^\pm(L_1,L_2) \ll
\frac{V^2}{d^2t}(\eta L)^3
=\frac{\eta VL^2}{rU}.
\]

We now estimate
\[
\Sigma_{ijr}^{\pm}(L_1,L_2) := \max_{L_1 \le L_1^\prime \le L_2}  \big| S_{ijr}^\pm(L_1^\prime,L_2)\big|
\]
where
\[
S_{ijr}^\pm(L_1^\prime,L_2) : =\sum_{\ell=L_1^\prime}^{L_2} e \left( g_i^\pm(\ell/r)-g_j^\pm(\ell/r) \right).
\]
In order to estimate the exponential sum $S_{ijr}^\pm(L_1^\prime,L_2)$,
we use van der Corput first and second derivative estimates in the form
of \cite[Lemmas~5.1.2 and 5.1.3]{Huxley}.  In particular, we need to
study the derivatives of the functions $g_{jr}^\pm(\ell)$.

\subsection{First derivative estimate}

Note that
\[
\frac{\mathrm{d}}{\mathrm{d}\ell} g_{jr}^\pm(\ell) = \frac{1}{r} (g_j^\pm)^\prime(\ell/r)  + \frac{\overline{ra_j}}{ q_j}.
\]
For the stationary point $x_j^\pm(\ell)$, we have
\[
-\frac{t}{2\pi x_j^\pm(\ell)} + \frac{u_j}{v_j} \mp \frac{1}{q_j}
\sqrt{\frac{\ell}{x_j^\pm(\ell)}}=0,
\]
so that
\[
\frac{\mathrm{d}}{\mathrm{d}\ell}  g_j^\pm(\ell) = \left\{
-\frac{t}{2\pi x_j^\pm(\ell)} + \frac{u_j}{v_j} \mp  \frac{1}{q_j}
\sqrt{\frac{\ell}{x_j^\pm(\ell)}} \right\}
\frac{\mathrm{d}x_j^\pm(\ell)}{\mathrm{d}\ell} \mp \frac{1}{q_j}
\sqrt{\frac{x_j^\pm(\ell)}{\ell}}
= \mp \frac{1}{q_j} \sqrt{\frac{x_j^\pm(\ell)}{\ell}}
\]
and
\begin{equation}\label{dg}
\frac{\mathrm{d}}{\mathrm{d}\ell} g_{jr}^\pm(\ell) = \mp \frac{1}{rq_j} \sqrt{\frac{x_j^\pm(\ell/r)}{\ell/r}} + \frac{\overline{r a_j }}{q_j}.
\end{equation}

\subsection{Second derivative estimate} We have
\[
\frac{\mathrm{d}^2}{\mathrm{d}\ell^2}  g_{jr}^\pm(\ell) =
\mp \frac{1}{rq_j} \frac{\mathrm{d}}{\mathrm{d}\ell}
\sqrt{\frac{x_j^\pm(\ell/r)}{\ell/r}}.
\]
Again writing $y_j=\frac{d}{2ru_j q_j}$, we have
\[
\frac{1}{rq_j}\sqrt{\frac{x_j^\pm(\ell/r)}{\ell/r}} =
\frac{d}{2ru_jq_j}\left(
\sqrt{1 + \frac{2tru_jq_j}{\pi d \ell} } \pm 1\right)=
y_j\left(\sqrt{1+\frac{t}{\pi y_j\ell}}\pm 1\right),
\]
and so
\[
\frac{\mathrm{d}^2}{\mathrm{d}\ell^2}  g_{jr}^\pm(\ell) =\mp \frac{1}{rq_j} \frac{\mathrm{d}}{\mathrm{d}\ell}
\sqrt{\frac{x_j^\pm(\ell/r)}{\ell/r}}
= \pm \frac{t}{2\pi \ell^2} \left( 1 + \frac{t}{\pi\ell y_j}
\right)^{\!-\frac12}.
\]
This is clearly a monotonic function of $y_j$. Therefore $(g_{ir}^\pm -
g_{jr}^\pm)^{\prime \prime}(\ell)$ is either identically zero, or it is
never zero, so  $(g_{ir}^\pm - g_{jr}^\pm)^{\prime}(\ell)$ is monotone
in $\ell$. Hence
\[
\begin{split}
\frac{\mathrm{d}^2}{\mathrm{d}\ell^2} \left[  g_{ir}^\pm(\ell) - g_{jr}^\pm(\ell) \right] &=
\pm \frac{t}{2\pi \ell^2} \left\{  \frac{1}{\sqrt{1+\frac{t}{\pi \ell y_i}}} - \frac{1}{\sqrt{1+\frac{t}{\pi \ell y_j} } } \right\}
\\
&= \pm \frac{t}{2\pi \ell^2} \left\{  \frac{\frac{t}{\pi \ell y_j} - \frac{t}{\pi \ell y_i} }{\sqrt{1+\frac{t}{\pi \ell y_i}}\sqrt{1+\frac{t}{\pi \ell y_j} } \left(\sqrt{1+\frac{t}{\pi \ell y_i}}+\sqrt{1+\frac{t}{\pi \ell y_j} } \, \right) } \right\}
\\
&= \pm \frac{r t^2}{\pi^2d\ell^3} \left\{  \frac{ u_j q_j - u_i q_i }{\sqrt{1+\frac{t}{\pi \ell y_i}}\sqrt{1+\frac{t}{\pi \ell y_j} } \left(\sqrt{1+\frac{t}{\pi \ell y_i}}+\sqrt{1+\frac{t}{\pi \ell y_j} } \, \right) } \right\}.
\end{split}
\]
As shown above, we have $\frac{t}{\pi \ell y_j} \gg 1$, so that
\[
\left| \frac{\mathrm{d}^2}{\mathrm{d}\ell^2} \left[  g_{ir}^\pm(\ell) - g_{jr}^\pm(\ell) \right] \right|
\asymp  \frac{rt^2}{dL^3} \frac{ \big| u_iq_i - u_jq_j
\big|}{(\frac{rtUV}{d^2L})^{\frac32}}
\asymp\frac{\eta}{L}\frac{|u_iv_i - u_jv_j|}{UV}.
\]

\subsection{Applying van der Corput estimates}
Our strategy is to use both van der Corput first and second
derivative estimates. If neither of these estimates is small, then
this implies constraints on the sizes of $|u_iv_i-u_jv_j|$ and
$\|\frac{\overline{ra_i}}{q_i}-\frac{\overline{ra_j}}{q_j}\|$. This leads
to the counting problem given in Proposition \ref{Large Sieve}. To this
end, define
\[
\begin{split}
N(X) &=\#\{(i,j)\in J(\beta,r)^2: (u_i,v_i),(u_j,v_j) \in \mathcal{R} \text{ and } \Sigma_{ijr}^{\pm}(L_1,L_2) \ge X\}.
\end{split}
\]
Trivially $\Sigma_{ijr}^{\pm}(L_1,L_2) \le  L$, so we have
\begin{equation} \label{X1}
\begin{split}
\sum_{\substack{(i,j)\in J(\beta,r)^2 \\ (u_i,v_i),(u_j,v_j) \in \mathcal{R} }} \!\!\!\!\!\!\! \Sigma_{ijr}^{\pm}(L_1,L_2)&=-\int_0^L X\, \mathrm{d}N(X)
=\int_0^L N(X)\, \mathrm{d}X
\\
&\le X_1 (\#\mathcal{R})^2+\int_{X_1}^L N(X)\, \mathrm{d}X,
\end{split}
\end{equation}
where we take $X_1:= A\sqrt{L}\max(\sqrt\eta,1)$ for a sufficiently large constant $A$. 

Suppose that $\Sigma_{ijr}^{\pm}(L_1,L_2) \ge X\ge X_1$.
If $u_iv_i-u_jv_j\ne0$ then we get a 
bound for $\Sigma_{ijr}^{\pm}$ by the second derivative test.
In particular, by \cite[Lemma~5.1.3]{Huxley}, we have
\[
X \le \Sigma_{ijr}^{\pm}(L_1,L_2) \ll
L\sqrt{\lambda_2} + \frac{1}{\sqrt{\lambda_2}} \ll
\sqrt{\eta L} + \frac{1}{\sqrt{\lambda_2}} \le
\frac{X}{A} + \frac{1}{\sqrt{\lambda_2}},
\]
where $\lambda_2 = \frac{\eta}{L}\frac{|u_iv_i - u_jv_j|}{UV}$.
If $A$ is large enough, then this implies that $1/\sqrt{\lambda_2} \gg X$ so therefore $\lambda_2 \ll X^{-2}.$ Thus, we have
\begin{equation}\label{upperboundestimate}
\big| u_iv_i - u_jv_j\big| \ll UV\frac{L}{\eta X^2},
\end{equation}
and obviously this holds also when $u_iv_i-u_jv_j=0$.

Next we apply a first derivative estimate. Let 
$z_{ijr}  = \frac{\overline{a_i r}}{q_i} - \frac{\overline{a_j r}}{q_j}$.
Then, with $y_j$ as above, from \eqref{dg} we derive that
\begin{equation} \label{zest}
\begin{split}
\Big|\frac{\mathrm{d}}{\mathrm{d}\ell} \big[ g_{ir}^\pm(\ell) - g_{jr}^\pm(\ell) \big] - z_{ijr}\Big| 
&= \left| \left(\sqrt{ y_i^2 + \frac{ty_i}{\pi \ell}} \pm y_i\right) - \left( \sqrt{ y_j^2 + \frac{ty_j}{\pi \ell}} \pm y_j \right) \right|
\\
&= |y_i - y_j| \left| \frac{ y_i + y_j + \frac{t}{\pi \ell} }{ \sqrt{ y_i^2 + \frac{ty_i}{\pi \ell}} + \sqrt{ y_j^2 + \frac{ty_j}{\pi \ell}} } \pm 1 \right|
\\
&\asymp |y_i - y_j| \sqrt{ \frac{rUVT}{d^2L} }\asymp\eta\frac{|u_iv_i-u_jv_j|}{UV}.
\end{split}
\end{equation}
Therefore, using the upper bound for $|u_iv_i-u_jv_j|$ in \eqref{upperboundestimate}, we have
\[
\frac{\mathrm{d}}{\mathrm{d}\ell}\left[ g_{ir}^\pm(\ell) - g_{jr}^\pm(\ell) \right] - z_{ijr}\ll \frac{L}{X^2} \le  \frac{L}{X_1^2} \ll 1.
\]
By (possibly) increasing the size of the constant $A$ in the definition
of $X_1$, we see that there exist numbers $\mu$ and $\nu$ such that
$\nu-\mu < \frac{1}{2}$ and 
$\mu \le \frac{\mathrm{d}}{\mathrm{d}\ell}\left[ g_{ir}^\pm(\ell) -
g_{jr}^\pm(\ell) \right] \le \nu$.
Now we use the truncated Poisson summation formula
\cite[Lemma~5.4.3]{Huxley}, which states that
\[
\sum_{\ell=L_1^\prime}^{L_2} e\big( g_{ir}^\pm(\ell) - g_{jr}^\pm(\ell) \big) = \sum_{\mu-\frac{1}{4}\le n \le \nu+\frac{1}{4}} \int_{L_1^\prime}^{L_2} e\big( g_{ir}^\pm(x) - g_{jr}^\pm(x) - nx \big) \, \mathrm{d}x +O(1).
\]
Note that the sum on the right-hand side contains at most one term. Hence
\[
\begin{split}
\Sigma_{ijr}^{\pm}(L_1,L_2) &\le
O(1) + \max_{L_1^\prime \in [L_1,L_2]}
\sum_{\mu-\frac{1}{4}\le n \le \nu+\frac{1}{4}}
\left|\int_{L_1^\prime}^{L_2} e\big( g_{ir}^\pm(x) - g_{jr}^\pm(x) - nx \big) \, \mathrm{d}x \right|
\\
&=O(1) + \sum_{\mu-\frac{1}{4}\le n \le \nu+\frac{1}{4}} \max_{L_1^\prime \in [L_1,L_2]} \left|  \int_{L_1^\prime}^{L_2} e\big( g_{ir}^\pm(x) - g_{jr}^\pm(x) - nx \big) \, \mathrm{d}x \right|.
\end{split}
\]
Since $z_{ijr}$ is only defined modulo $1$, we are free to shift
$g_{ir}^\pm(\ell) - g_{jr}^\pm(\ell)$ by any integer multiple of
$\ell$. Thus, if there is an integer $n \in
[\mu-\frac{1}{4},\nu+\frac{1}{4}]$, we may assume that $n=0$.
Therefore,
\[
\Sigma_{ijr}^{\pm}(L_1,L_2) \le
\max_{L_1^\prime \in [L_1,L_2]} \left| \int_{L_1^\prime}^{L_2} e\big( g_{ir}^\pm(x) - g_{jr}^\pm(x) \big) \, \mathrm{d}x \right| + O(1).
\]
Define
\[
\lambda_1 = \min_{\ell\in[L_1,L_2]} \left| \frac{\mathrm{d}}{\mathrm{d}\ell}\left[ g_{ir}^\pm(\ell) - g_{jr}^\pm(\ell) \right] \right|.
\]
Then by \cite[Lemma~5.1.2]{Huxley} we have
\[
\max_{L_1^\prime \in [L_1,L_2]} \left| \int_{L_1^\prime}^{L_2} e\big( g_{ir}^\pm(x) - g_{jr}^\pm(x) \big) \, \mathrm{d}x \right|  \ll \frac{1}{\lambda_1}.
\]
This implies that
$X \le \Sigma_{ijr}^{\pm}(L_1,L_2) \ll \frac{1}{\lambda_1} + 1$,
so that
$\lambda_1 \ll \frac{1}{X}$
if the constant $A$ is sufficiently large. Thus, by \eqref{zest}, we find that
\begin{equation}\label{zbound2}
\begin{split}
\|z_{ijr}\| &\le |z_{ijr}| = \lambda_1 +
O\!\left(\eta\frac{|u_iv_i - u_jv_j|}{UV}\right)
\ll \frac{1}{X} + \frac{L}{X^2} \ll \frac{L}{X^2},
\end{split}
\end{equation}
since $X \le L$.

In summary, we have found that
$\Sigma_{ijr}^\pm(L_1,L_2) \ge X \ge X_1$ implies the inequalities
\eqref{upperboundestimate} and \eqref{zbound2}. In other words,
$N(X) \le B(\Delta_1(X),\Delta_2(X))$
for certain functions $\Delta_1(X)$ and $\Delta_2(X)$ satisfying the
conditions in \eqref{DeltaConditions}. From \eqref{X1}, we derive that
\[
\sum_{\substack{(i,j)\in J(\beta,r)^2 \\ (u_i,v_i),(u_j,v_j) \in \mathcal{R} }} \!\!\!\!\!\!\! \Sigma_{ijr}^{\pm}(L_1,L_2)\le X_1 (\#\mathcal{R})^2+\int_{X_1}^L  B(\Delta_1(X)\Delta_2(X))\, \mathrm{d}X.
\]
By extending the definitions of $\Delta_1(X)$ and $\Delta_2(X)$ to be zero for $X< X_1$, we note that  the right-hand side of this expression is
\[
\ll X_0 (\#\mathcal{R})^2+\int_{X_0}^L B(\Delta_1(X),\Delta_2(X))\, \mathrm{d}X
\]
for $X_0$ defined in Proposition \ref{Large Sieve} so long as $A\ge 1$. The proposition now follows. \\

\noindent {\it Acknowledgements}.
We thank Peter Sarnak for helpful comments and the Banff International
Research Station for hosting us for a Research in Teams Meeting
(15rit201). A significant portion of this project was completed
during that week and we appreciated the excellent working conditions.
The second and third authors also thank the University of Bristol for
hosting a number of research visits.

\end{document}